\documentclass[11pt]{amsart}
\usepackage[margin=3.2cm]{geometry}
\usepackage{latexsym}
\usepackage{amsmath}
\usepackage{amssymb}
\usepackage{amsthm}
\usepackage{enumerate}
\usepackage{enumitem}
\newcommand{\RN}[1]{%
  \textup{\uppercase\expandafter{\romannumeral#1}}%
}
\newcommand{\rn}[1]{%
  \textup{\lowercase\expandafter{\romannumeral#1}}%
}

\newtheorem{theorem}{Theorem}[section]
\newtheorem*{cj1*}{Conjecture A}
\newtheorem*{cj2*}{Conjecture B}
\newtheorem*{cj3*}{Conjecture C}
\newtheorem*{cj4*}{Conjecture D}
\newtheorem*{cj5*}{Conjecture E}

\newtheorem{lemma}[theorem]{Lemma}
\newtheorem{conjecture}[theorem]{Conjecture}
\newtheorem{corollary}[theorem]{Corollary}
\newtheorem{question}[theorem]{Question}

\newtheorem{definition}[theorem]{Definition}

\newtheorem{rmk}[theorem]{Remark}

\newtheorem{example}[theorem]{Example}

\usepackage{todonotes}
\usepackage{hyperref}
\usepackage{enumitem}
\makeatletter
\newcommand*{\rom}[1]{\expandafter\@slowromancap\romannumeral #1@}
\newcommand{\upperRomannumeral}[1]{\uppercase\expandafter{\romannumeral#1}}

 \newcommand{\Df}{\text{Diff}_+}
  \newcommand{\Dff}{\text{Diff}_0}
    \newcommand{\Dfh}{\text{Diff}_h}
\newcommand{\ZZ}{\mathbb{Z}}

\newcommand{\CC}{\mathbb{C}}
\newcommand{\RR}{\mathbb{R}}
\newcommand{\PP}{\mathbb{P}}

\newcommand{\cS}{\mathcal{S}}

\newcommand{\dd}{\delta}

\title{The spaces of K\"ahler and holomorphically tamed symplectic forms on closed 4-manifolds}
\author{Tian-Jun Li and Shengzhen Ning}
\dedicatory{In memory of Professor Dietmar Salamon}
\AtEndDocument{\bigskip{\footnotesize
		\textsc{School of Mathematics, University of Minnesota, Minneapolis, MN, US} \par
		\textit{E-mail address}: \texttt{tjli@math.umn.edu} \par
		\addvspace{\medskipamount}
		\textsc{School of Mathematics, University of Minnesota, Minneapolis, MN, US} \par
		\textit{E-mail address}: \texttt{ning0040@umn.edu} \par
		
}}
\date{}

\begin{document}
\begin{abstract}
This paper investigates the uniqueness, connectedness and openness properties of spaces of K\"ahler forms on closed $4$-manifolds, extending discussions in \cite{Li08,Salamonsurvey}. Motivated by the Streets--Tian conjecture concerning the existence of K\"ahler metrics on Hermitian-symplectic complex manifolds, we also study holomorphically tamed symplectic forms, which are symplectic forms tamed by some integrable complex structure. We formulate a parallel question and relate it to the corresponding questions for K\"ahler-type symplectic forms.
\end{abstract}
\maketitle

\tableofcontents

\section{Introduction}

\subsection{A symplectic version of Streets-Tian conjecture}
In the context of complex differential geometry, a \emph{Hermitian-symplectic metric} on a complex manifold $(X,J)$ refers to a Hermitian metric $g$ whose associated fundamental $2$-form $g(J\cdot,\cdot)$ is the $(1,1)$-part of a symplectic form $\omega$. A complex manifold is said to be \emph{Hermitian-symplectic} if it admits a Hermitian-symplectic metric. This notion, which plays an important role in the study of curvature flows, was introduced in \cite{ST10} together with the Streets--Tian conjecture.

\begin{conjecture}[Streets--Tian]\label{conj:ST}
Any closed Hermitian-symplectic complex manifold $(X,J)$ admits a K\"ahler metric.
\end{conjecture}

 The conjecture was known to be true in the case of complex surfaces (\cite{ST10,LZ09}). For higher dimensions, it has since been extensively studied and verified in many cases. We refer to \cite{Zheng1,Zheng2,CaoZheng} for recent progress and an overview of the current status of the conjecture. In real dimension $4$, Donaldson's tame-to-compatible question
\cite{Donaldson} can be viewed as the almost K\"ahler version of the above conjecture: $J$ is only assumed to be an almost complex structure which is not necessarily integrable.
    
From the viewpoint of symplectic geometry, a Hermitian-symplectic complex manifold can be equivalently thought of as a symplectic manifold $(X,\omega)$ equipped with a tamed complex structure $J$ which means that $\omega(v,Jv)>0$ for any $v\neq 0$. In this paper, we investigate the following notions for symplectic forms.

\begin{definition}
    A symplectic form $\omega$ on $X$ is called \textbf{holomorphically tamed} (resp. {\bf of K\"ahler type}) if there exists a complex structure $J$ on $X$ which is tamed (resp. compatible) with $\omega$.
\end{definition}

 Parallel to Streets--Tian's conjecture \ref{conj:ST}, one could naively ask the analogous question of whether any holomorphically tamed symplectic form must also be of K\"ahler type, by fixing the symplectic structure instead of the complex structure.   Unfortunately, the answer to this question is negative even at the cohomological level. Let us recall some notions introduced in \cite{LZ09}. Given any almost complex manifold $(X,J)$, its \emph{$J$-tamed symplectic cone} and \emph{$J$-compatible symplectic cone} are the regions in $H^2(X;\RR)$ defined as
\begin{align*}
    \mathcal{K}_J^t(X)=\{[\omega]\in H^2(X;\RR)\,|\,\omega \text{ is a symplectic form tamed by }J\},\\
\mathcal{K}_J^c(X)=\{[\omega]\in H^2(X;\RR)\,|\,\omega \text{ is a symplectic form compatible with }J\}.
\end{align*}
The \emph{integrable tamed cone} $\mathcal{K}^t_{\text{Int}}(X)$ and \emph{integrable compatible cone} $\mathcal{K}^c_{\text{Int}}(X)$ are defined to be the union of $\mathcal{K}_J^t(X)$ and $\mathcal{K}_J^c(X)$ respectively, over all integrable $J$.

\begin{example}\label{example:Tedi}
Dr\u{a}ghici observed in \cite{Tedi} that a minimal K\"ahler surface of
general type $(X,J,\omega)$ with $b_2^+(X)>1$, or equivalently
$p_g(X)>0$, admits holomorphically tamed symplectic forms that are not of
K\"ahler type.
Choose a nonzero holomorphic $(2,0)$-form $\phi$ and define
\(
\omega_t:=\omega+t\operatorname{Re}(\phi).
\)
One can check
\(
\operatorname{Re}(\phi)(v,Jv)=0
\)
for every tangent vector $v$ so that $\omega_t$ tames $J$ for every
$t\in\RR$.
On the other hand,
\(
K_J\cdot[\omega_t]=K_J\cdot[\omega]
\)
is independent of $t$, whereas $[\omega_t]^2$ tends to infinity as
$|t|\to\infty$. The Hodge index theorem therefore implies that, for
$|t|$ sufficiently large, the classes $K_J$ and $[\omega_t]$ cannot
both be of type $(1,1)$ with respect to the same complex structure.
By \cite[Corollary~7.4.3]{Morganbook}, the canonical class of any complex
structure $J'$ on $X$ satisfies
\(
K_{J'}=\pm K_J.
\)
Consequently, for $|t|$ sufficiently large,
\(
[\omega_t]\notin H_{J'}^{1,1}(X;\RR)
\)
for every complex structure $J'$ on $X$. Hence $\omega_t$ cannot be of
K\"ahler type, although it is tamed by $J$.
\end{example}

Nevertheless, there are still many cases in which the integrable compatible and tamed cones are known to coincide. For example, when $b_2^+(X)=1$, the equality
\(
\mathcal{K}_{\operatorname{Int}}^c(X)
=
\mathcal{K}_{\operatorname{Int}}^t(X)
\)
was known in \cite{LZ09} (Theorem \ref{thm:LZb1}). The same equality for $X=K3$ and
$X=\mathbb{T}^4$ was observed in \cite{Li08}.
These considerations motivate us to formulate the following question.

\begin{question}[Symplectic version of Streets--Tian conjecture]
\label{question:sympST}
Let $\omega$ be a holomorphically tamed symplectic form on a closed manifold $X$ such that
\(
[\omega]\in\mathcal{K}_{\operatorname{Int}}^c(X).
\)
Does there exist a complex structure $J$ on $X$ that is compatible with
$\omega$?
\end{question}

\subsection{Space of symplectic forms}\label{section:spacesymp}

Let $X$ be a closed oriented $4$-manifold. Denote by
\(
\mathcal S(X)\subseteq Z(X)
\)
the spaces of symplectic forms and closed $2$-forms
on $X$, respectively, both equipped with the $C^\infty$-topology. The
topology of $\mathcal S(X)$ provides an interesting invariant of the
underlying manifold $X$. It was shown in
\cite[Lemma~2.2]{Li08} that $\mathcal S(X)$ is an open and locally
convex subset of $Z(X)$.

Let $\Df(X)$ be the group of orientation-preserving diffeomorphisms of
$X$, and let $\Dff(X)$ denote its identity component. The mapping class
group of $X$ is
\[
\Gamma(X):=\Df(X)/\Dff(X).
\]
Let $\Dfh(X)\subseteq\Df(X)$ be the subgroup consisting of
diffeomorphisms acting trivially on homology. The Torelli group of $X$
is defined by
\[
\Gamma_h(X):=\Dfh(X)/\Dff(X).
\]
One may consider the quotient
\[
\widetilde{\mathcal M}(X):=\mathcal S(X)/\Dff(X).
\]
The space $\widetilde{\mathcal M}(X)$ and its discrete quotient
\[
\mathcal M(X)
:=
\widetilde{\mathcal M}(X)/\Gamma(X)
=
\mathcal S(X)/\Df(X)
\]
can be viewed as non-Hausdorff ``manifolds'' of dimension
$b_2(X)$; see \cite{FH02,FHH05,Wilsonremark,Li08}. When $X$ is a ruled $4$-manifold, the classification of symplectic forms due to Lalonde--McDuff  \cite{LalondeMcduff96} implies that $\mathcal{M}(X)$ is connected. In general, however,
the topology of $\widetilde{\mathcal M}(X)$ and $\mathcal M(X)$ remains
poorly understood. Even determining their numbers of connected
components is a subtle problem.

One can refine this question by fixing a cohomology class. Let
$a\in H^2(X;\RR)$ be represented by a symplectic form $\omega$. Denote
by
\(
\mathcal S_a(X)
\)
the space of all symplectic forms representing $a$, and by
\(
\mathcal S_\omega(X)\subseteq\mathcal S_a(X)
\)
the connected component containing $\omega$. A fundamental question in
$4$-dimensional symplectic topology
\cite[Problem~2(a)]{MSbook} asks whether
\(
\mathcal S_a(X)=\mathcal S_\omega(X).
\)  
Only very recently did \cite{LWXZ} show that $\mathcal S_a(X)$ can have
infinitely many connected components when $X$ is an $S^2$-bundle over a Riemann surface of
positive genus. By Moser's stability, this is equivalent to the failure of the $\Dff(X)$-action on $\mathcal S_a(X)$ to be transitive, whereas $\Dfh(X)$ acts transitively on $\mathcal S_a(X)$ by \cite{LalondeMcduff96}.

\subsection{Space of K\"ahler forms}

Suppose that $X$ is the underlying oriented smooth $4$-manifold of a
K\"ahler surface. We consider the following subspaces of
$\mathcal S(X)$:
\begin{align*}
    \mathcal{SK}(X)
    &:=
   \{
    \omega\in\mathcal S(X)
    \,|\,
    \omega \text{ is of K\"ahler type}
    \},\\
    \mathcal{ST}(X)
    &:=
    \{
    \omega\in\mathcal S(X)
    \,|\,
    \omega \text{ is holomorphically tamed}
    \}.
\end{align*}
These spaces can be expressed as
\[
\mathcal{SK}(X)
=
\bigcup_{J}\mathcal{SK}_J(X),
\qquad
\mathcal{ST}(X)
=
\bigcup_{J}\mathcal{ST}_J(X),
\]
where
\begin{align*}
    \mathcal{SK}_J(X)
    &:=
    \{
    \omega\in\mathcal S(X)
    \,|\,
    \omega \text{ is compatible with }J
    \},\\
    \mathcal{ST}_J(X)
    &:=
   \{
    \omega\in\mathcal S(X)
    \,|\,
    \omega \text{ tames }J
    \},
\end{align*}
and the union is over the space of all complex structures on $X$.
For a class $a\in H^2(X;\RR)$, define
\[
\mathcal{SK}_a(X)
:=
\{
\omega\in\mathcal{SK}(X)
\,|\,
[\omega]=a
\},
\qquad
\mathcal{ST}_a(X)
:=
\{
\omega\in\mathcal{ST}(X)
\,|\,
[\omega]=a
\}.
\]
%Given $\omega\in\mathcal{SK}_a(X)$, let
%\[
%\mathcal{SK}_\omega(X)\subseteq\mathcal{SK}_a(X)
%\]
%denote the connected component containing $\omega$. Similarly, for
%$\omega\in\mathcal{ST}_a(X)$, let
%\[
%\mathcal{ST}_\omega(X)\subseteq\mathcal{ST}_a(X)
%\]
%denote the connected component containing $\omega$.
In this notation, Question~\ref{question:sympST} asks whether
\(
\mathcal{ST}_a(X)=\mathcal{SK}_a(X)
\)
whenever $\mathcal{SK}_a(X)$ is known to be nonempty.

We also consider the quotient spaces
\[
\mathcal{MT}(X)
:=
\mathcal{ST}(X)/\Df(X),
\qquad
\mathcal{MK}(X)
:=
\mathcal{SK}(X)/\Df(X),
\]
when the cohomology class is not fixed. For a fixed class
$a\in H^2(X;\RR)$, we define
\[
\mathcal{MK}_a(X)
:=
\mathcal{SK}_a(X)/\Dfh(X).
\]
Based on the results reviewed in
Section~\ref{section:spacesymp} about the space of symplectic forms, one may ask the following analogous questions
concerning the space of K\"ahler type forms.

\begin{question}[Uniqueness]\label{question:unique}
    Is $\mathcal{MK}_{a}(X)$ a single point? 
\end{question}

\begin{question}[Connectedness]\label{question:connected}
    Is $\mathcal{MK}(X)$ connected?
\end{question}

\begin{question}[Openness]\label{question:open}
    Is the inclusion $\mathcal{SK}(X)\hookrightarrow \mathcal{S}(X)$ open?
\end{question}

 Questions \ref{question:unique} and \ref{question:connected} are the K\"ahler-type incarnations of the insightful questions raised by Salamon in \cite[Question 2 and 3]{Salamonsurvey}. We aim to verify whether the (non-)uniqueness and deformation properties pursued by him for all symplectic forms can hold after restricting to K\"ahler-type symplectic forms. Recall that the classical Kodaira--Spencer stability theorem \cite{KSstab} asserts the K\"ahler stability of complex structures: a small complex deformation of K\"ahler manifolds remains K\"ahler. In this sense, Question \ref{question:open} may be viewed as the symplectic analogue of Kodaira--Spencer stability asking about the K\"ahler stability of symplectic structures.

\subsection{Main results}\label{section:main results}
We now state the main results of this paper. Let $X$ be a closed,
oriented, smooth $4$-manifold. Concerning
Question~\ref{question:sympST}, we have the following.

\begin{theorem}\label{thm:main1}
For every $a\in H^2(X;\RR)$, the space $\mathcal{SK}_a(X)$ is either
empty or equal to $\mathcal{ST}_a(X)$ under any of the following
assumptions:
\begin{enumerate}
    \item $b_2^+(X)=1$;
    \item $X$ admits a symplectic form of Kodaira dimension zero;
    \item $X$ admits a K\"ahler metric of negative sectional curvature.
\end{enumerate}
\end{theorem}

The three cases may be viewed very roughly as three regimes containing K\"ahler surfaces with positive, zero, and negative curvature respectively, and their proofs require different strategies. Item~(1) includes all rational and ruled $4$-manifolds and follows from the comparison between tamed and compatible cones established in \cite{LZ09}. Item~(2) includes K3 surfaces, Enriques surfaces, hyperelliptic surfaces, complex $2$-tori $\mathbb T^4$, and their blowups; its proof uses the classification and Torelli theory of compact K\"ahler surfaces of Kodaira dimension zero \cite{BHPVbook}. Finally, item~(3) includes all compact complex ball quotients \( \Pi\backslash\mathbb B_{\CC}^2\) for some \(\Pi<\operatorname{PU}(2,1), \) whose proof relies on strong rigidity results on negatively curved manifolds \cite{Siurigidity,Zhengrigidity}.

We next turn to Question \ref{question:unique}.
\begin{theorem}\label{thm:main2}
 For every $a\in H^2(X;\RR)$, we have $\#\mathcal{MK}_a(X)\leq 1$ under either of the following assumptions:
 \begin{enumerate}
     \item  $X$ admits a symplectic form of non-positive Kodaira dimension;
     \item $X$ admits a K\"ahler metric of negative sectional curvature.
 
 \end{enumerate}
  Moreover, if $X$ underlies a K\"ahler surface of general type with $b_2^+(X)=1$, then $\#\mathcal{MK}_a(X)$ is finite.
\end{theorem}

When $X$ is a K3 surface, the uniqueness result was proved by Tolman in
\cite[Proposition~2.1]{Tolman} in order to construct her beautiful example of a
$6$-dimensional non-Hamiltonian symplectic circle action with $32$ fixed
points. Entov--Verbitsky \cite{EV16,EV23} also established such
uniqueness for $\mathbb T^{2n}$. The main novelty of item~(1)
above is to extend the uniqueness result to their blowups, as well as Enriques surfaces and hyperelliptic surfaces, through
holomorphically tamed symplectic forms. Indeed, many results concerning the symplectic topology of Calabi--Yau surfaces and their blowups assume that the symplectic form is of
K\"ahler type, or even projective. For results on symplectomorphism
groups, see \cite{SSk3,SmirnovK3,smirnovT4}; for symplectic packings,
see \cite{LMS13,EV16,EV18,EVlong}; and for Lagrangian submanifolds/fibrations and mirror symmetry, see \cite{AST4,SSk3,EV23,CG25,hackingkeating25}.
We expect that Theorem~\ref{thm:main2} combined with Theorem \ref{thm:main1} is useful by weakening the hypothesis on the symplectic form from
being K\"ahler to being merely holomorphically tamed; and once a
symplectic property is established for one holomorphically tamed
symplectic form,
Theorem~\ref{thm:main2} allows it to be transferred to every
cohomologous holomorphically tamed symplectic form. This principle has already been employed in \cite{EV23}, which allows dynamical results about diffeomorphism group actions to be applied to the study of Lagrangian tori.

Let us now address Question \ref{question:connected}.
\begin{theorem}\label{thm:main3}
 The space $\mathcal{MK}(X)$ 
    \begin{enumerate}
    \item has finitely many connected components for any $X$;
    \item is connected, if $X$ admits a symplectic form of non-positive Kodaira dimension;
      \item has at most two connected components, if $X$ admits a K\"ahler metric of negative sectional curvature.
    \end{enumerate}
\end{theorem}

Fake projective planes (\cite{Mumfordfake}) provide examples for which the bound two in item~(3) is attained. This would follow from the nonexistence of anti-holomorphic involutions established in \cite{KK02}.
The finiteness statements in ~\ref{thm:main3} as well as Theorem ~\ref{thm:main2} (1)
rely on some general constraints from gauge theory. Establishing connectedness
and uniqueness, however, requires effort to have a more precise understanding of the
moduli of complex structures on these manifolds in Section \ref{section:k3}.

Finally, for Question \ref{question:open}, we have:
\begin{theorem}\label{thm:main4}
     $\mathcal{SK}(X)\hookrightarrow\mathcal{S}(X)$ is open under either of the following assumptions:
     \begin{enumerate}
    \item $b_2^+(X)=1$;
    \item $X$ admits a symplectic form of Kodaira dimension zero.
\end{enumerate}
 Moreover, if $b_2^+(X)>1$ and $X$ admits a K\"ahler metric of negative sectional curvature, then  $\mathcal{SK}(X)\hookrightarrow\mathcal{S}(X)$ is not open.
\end{theorem}

 Using geometric flow techniques, Fei--Phong--Picard--Zhang \cite{FPPZ}
and Streets--Tian \cite{StT} proved that, for a K\"ahler manifold
$(X,J,\omega)$ of arbitrary dimension such that $c_1(X,J)=0$, the
inclusion
\(
\mathcal{SK}(X)\hookrightarrow\mathcal{S}(X)
\)
is open at $\omega$. In dimension four, the theorem above applies to a
broader class of manifolds, including blowups of Calabi--Yau surfaces
and other manifolds with $b_2^+=1$. Our proof of openness
in dimension four does not rely on geometric flows. For K3 and $\mathbb T^4$, an openness argument was outlined by Entov--Verbitsky in \cite[Remark 1.3]{EV23} and \cite[Example 2.1]{EVlong}. Our main effort here is to establish the openness for their blowups. 
On the other hand, the above theorem also provides examples for which openness fails. For example, the Cartwright--Steger surface \cite{CSsurface} is a certain complex ball quotient with numerical invariants
\(
p_g=1, h^{1,1}=3
\). Its underlying smooth manifold has $b_2^+=3$ so that the space of K\"ahler-type forms is not open in the space of symplectic forms.

\medskip
\noindent\textbf{Organization of the paper.}
Section \ref{section:prep} collects the background material used throughout the paper. In Section \ref{section:b+1}, we treat the case $b_2^+=1$, where the deformation-to-isotopy theorem allows us to establish our results by relatively direct arguments. The technical core of the paper is Section \ref{section:k3}, where we analyze K3 surfaces, complex tori, and their blowups via the Torelli theorem and the topology of the period domain. Finally, in Section \ref{section:proof}, we assemble these ingredients to prove the main theorems stated in Section \ref{section:main results}.

\medskip
\noindent\textbf{Acknowledgments.}
The authors are supported by Simons Travel Fund SFI-MPS-TSM-00013390. We wish to thank   Tedi Draghici,  Yi Du, Hao Fang, Cheuk Yu Mak, Adriano Tomassini, Xiaokui Yang, Saikee Yeung, Weiyi Zhang and Fangyang Zheng for useful discussions about various aspects of this paper over the years. 

\section{Preparations}\label{section:prep}
\subsection{Cohomological aspects}

We begin by recalling some cohomological results. Let $(X,J)$ be a
closed almost complex $4$-manifold. \cite{LZ09} considered the
decomposition
\[
H^2(X;\RR)=H_J^+(X)\oplus H_J^-(X),
\]
where $H_J^+(X)$ and $H_J^-(X)$ consist of the cohomology classes that
can be represented by closed $J$-invariant and $J$-anti-invariant real
$2$-forms, respectively. The $J$-compatible cone
$\mathcal K_J^c(X)$ is naturally contained in $H_J^+(X)$.
When $J$ is integrable, they agree with the real parts of the usual Dolbeault cohomology groups
\begin{align*}
H_J^+(X)
&=
H^{1,1}_J(X;\RR)=H^{1,1}_{\overline{\partial}}(X)\cap H^2(X;\RR),\\
H_J^-(X)
&=(
H_{J}^{2,0}(X)
\oplus
H_{J}^{0,2}(X)
)_\RR=
(
H_{\overline{\partial}}^{2,0}(X)
\oplus
H_{\overline{\partial}}^{0,2}(X)
)
\cap H^2(X;\RR).
\end{align*}

\begin{theorem}\label{thm:LZb1}
Let $(X,J)$ be a closed almost complex $4$-manifold.
\begin{enumerate}
    \item (\cite{LZ09,ST10}) If $J$ is integrable, then
    $\mathcal K_J^c(X)$ is empty if and only if
    $\mathcal K_J^t(X)$ is empty.

    \item (\cite{LZ09}) If $\mathcal K_J^c(X)$ is nonempty, then
    \(
    \mathcal K_J^t(X)
    =
    \mathcal K_J^c(X)+H_J^-(X)
    \).
\end{enumerate}
\end{theorem}
Since this paper concerns K\"ahler forms and holomorphically tamed symplectic forms, we
will primarily apply the integrable case of the above theorem. 
By considering the map \[c:Z(X)\rightarrow H^2(X;\RR)\] which sends a closed 2-form to its cohomology class, compatible and tamed cones can be expressed as the images under the map $c$:
\begin{align*}
\mathcal{K}_J^c(X)&=c(\mathcal{SK}_J(X)),\quad \mathcal{K}_J^t(X)=c(\mathcal{ST}_J(X)),\\
\mathcal{K}_{\text{Int}}^c(X)&=c(\mathcal{SK}(X)),\quad \mathcal{K}_{\text{Int}}^t(X)=c(\mathcal{ST}(X)).
\end{align*}
 Clearly, we have
$$\mathcal K^c_{\text{Int}}(X) \subseteq \mathcal K^t_{\text{Int}}(X) \subseteq c(\mathcal{S}(X)),$$
where the rightmost set is called the symplectic cone. Since nondegeneracy and tameness with respect to a fixed complex structure are open conditions, both the integrable tamed cone and symplectic cone are open subsets of \(H^2(X;\RR)\).
Example~\ref{example:Tedi} shows that the first inclusion can be strict when $b_2^+(X)>1$. On the other hand, \cite{CP12,NingJLMS} provide examples with $b_2^+(X)=1$ for which \(\mathcal K_{\operatorname{Int}}^c(X) \subsetneq c(\mathcal S(X)). \) By Theorem~\ref{thm:LZb1}, one has \( \mathcal K_{\operatorname{Int}}^c(X) = \mathcal K_{\operatorname{Int}}^t(X) \) whenever $b_2^+(X)=1$. Consequently, in these examples the second inclusion is also strict: \( \mathcal K_{\operatorname{Int}}^t(X) \subsetneq c(\mathcal S(X)). \)

\begin{lemma}\label{lem: continuity of the class map}
    The cohomology class map $c$ is continuous.
\end{lemma}
\begin{proof}
    One can choose a basis in $H_2(X;\RR)$ represented by embedded surfaces $C_1,\cdots,C_r$. Then, the map $c$ can be interpreted as period map $\alpha\mapsto(\int_{C_1}\alpha,\cdots,\int_{C_r}\alpha)\in \RR^d$. Since each integration depends continuously on $\alpha$, the map $c$ is also continuous.
\end{proof}

\subsection{Tamed and compatible forms along a path of almost complex structures}

In this section, we consider the union of tamed and compatible symplectic forms along paths of (almost) complex structures. For tamed forms, we have the following.

\begin{lemma}\label{lem:isotopic}
Let $\{J_t\}_{t\in[0,1]}$ be a path of almost complex structures on an oriented closed smooth $4$-manifold $X$ such that $\mathcal{ST}_{J_t}(X)\neq\varnothing$ for every $t$. Then
 \(
    \bigcup_{t\in[0,1]}\mathcal{ST}_{J_t}(X)
    \)
    is connected.
     Moreover, for any class \(
\kappa\in\bigcap_{t\in[0,1]}\mathcal K_{J_t}^t(X)
\), any two symplectic forms in \[\mathcal{ST}_\kappa(X)\bigcap\Big(\bigcup_{t\in[0,1]}\mathcal{ST}_{J_t}(X)\Big)\] are related by a diffeomorphism in $\Dff(X)$.
\end{lemma}

\begin{proof}
Let $\omega_0$ and $\omega_1$ be symplectic forms tamed by $J_0$ and $J_1$ respectively. For each $t\in[0,1]$, choose
a $J_t$-tamed symplectic form $\omega_t$, taking
the prescribed forms at $t=0$ and $t=1$.
Since tameness is an open condition, for every $t\in[0,1]$ there exists
an open neighborhood $U_t$ of $t$ such that $\omega_t$ tames $J_s$ for
every $s\in U_t$. By compactness of $[0,1]$, we may choose a finite
sequence
\(
0=t_0<t_1<\cdots<t_N=1
\)
such that
\(
t_{j+1}\in U_{t_j}
\)
for every $j=0,\ldots,N-1$.
For each $j$, both $\omega_{t_j}$ and $\omega_{t_{j+1}}$ tame
$J_{t_{j+1}}$. Therefore, their linear interpolation
\(
\omega_{j,s}:=(1-s)\omega_{t_j}+s\omega_{t_{j+1}}
\), where $s\in[0,1]$,
also tames $J_{t_{j+1}}$, since the space of $J_{t_{j+1}}$-tamed forms
is convex. Concatenating these finitely many paths gives a path from
$\omega_0$ to $\omega_1$ through symplectic forms. Thus, \(
    \bigcup_{t\in[0,1]}\mathcal{ST}_{J_t}(X)
    \)
    is connected.

If \(
\kappa\in\bigcap_{t\in[0,1]}\mathcal K_{J_t}^t(X)
\) and $[\omega_0]=[\omega_1]=\kappa$, we may choose $\omega_t$ such that $[\omega_t]=\kappa$. This implies
\(
[\omega_{j,s}]=\kappa
\)
for every $s$, so $\omega_{j,s}$ is a symplectic form representing
$\kappa$. Thus, by Moser's stability, we have $\omega_0=f^*\omega_1$ for some $f\in\Dff(X)$.
\end{proof}

A version of this lemma also appears in \cite[Lemma~2.10]{Tolman}, where it was proved using a partition of unity argument. We now adapt that argument to the case of symplectic forms of K\"ahler type.

\begin{lemma}\label{lemma:uniontamed}
Let $\{J_t\}_{t\in[0,1]}$ be a path of integrable complex structures on an oriented closed smooth $4$-manifold $X$. Then
    \(
    \bigcup_{t\in[0,1]}\mathcal{SK}_{J_t}(X)
    \)
    is connected.
\end{lemma}

\begin{proof}
 Assume \(
    \bigcup_{t\in[0,1]}\mathcal{SK}_{J_t}(X)
    \) is nonempty. Let
\(
\omega_{t_0}\in\mathcal{SK}_{J_{t_0}}(X),
\omega_{t_1}\in\mathcal{SK}_{J_{t_1}}(X).
\)
Since $(X,J_{t_0})$ is K\"ahler, $b_1(X)$ is even. By the
topological characterization of compact K\"ahler surfaces
\cite{Buchdahl,Lamari}, every complex surface $(X,J_t)$ is therefore
K\"ahler.

For each $s\in[t_0,t_1]$, choose a $J_s$-K\"ahler form $\omega_s$,
taking the prescribed forms at $s=t_0,t_1$. By the Kodaira--Spencer
stability theorem \cite[Theorem~15]{KSstab}, there exists an open
neighborhood
\(
U_s\subseteq[t_0,t_1]
\)
of $s$ and a smooth family of $J_t$-K\"ahler forms
$\omega_s(t)$ for $t\in U_s$,
such that
\(
\omega_s(s)=\omega_s.
\)

Now, choose a finite subcover
\(
U_{s_0},\ldots,U_{s_N}
\)
with $s_0=t_0$ and $s_N=t_1$, and a smooth partition of unity
$\{\rho_0(t),\ldots,\rho_N(t)\}\subseteq C^{\infty}([t_0,t_1],[0,1])$ subordinate to this cover satisfying
\(
\rho_0(t_0)=1,
\rho_N(t_1)=1.
\)
Set
\[
\Omega_t:=\sum_{j=0}^N\rho_j(t)\omega_{s_j}(t).
\]
For each $t$, the form $\Omega_t$ is a positive linear combination of finitely many $J_t$-K\"ahler forms $\{\omega_{s_j}(t)\}_{j=1}^N$. Therefore, $\Omega_t$ is also $J_t$-K\"ahler. Furthermore,
\(
\Omega_{t_0}=\omega_{t_0},
\Omega_{t_1}=\omega_{t_1}.
\)
Hence,
\(
\bigcup_{t\in[0,1]}\mathcal{SK}_{J_t}(X)
\)
is connected.
\end{proof}

\begin{rmk}
The integrability assumption in Lemma \ref{lemma:uniontamed} is essential to the argument, as it allows us to apply the Kodaira--Spencer stability theorem. An analogous stability result for almost K\"ahler structures is not known in general. Indeed, in real dimension at least six, any K\"ahler complex structure can be included in a path of almost complex structures such that every sufficiently small nontrivial deformation is not even locally compatible with any symplectic form; see \cite{MT00,HMT23}. In dimension four, this is also related to Donaldson's tame-to-compatible question. 
\end{rmk}

\subsection{Rigidity of K\"ahler surfaces with negative sectional curvature}

We recall the strong rigidity result that will be used later. The
marked form of rigidity is important for our purposes, since we need
the resulting biholomorphism to act trivially on cohomology.

\begin{definition}\label{def:rigid}
A compact K\"ahler manifold $(X,J,g)$ is said to be marked
strongly rigid if, for every compact K\"ahler manifold $(Y,J',g')$ and
every homotopy equivalence
\(
h:Y\rightarrow X
\),
there exists a biholomorphism or anti-biholomorphism
\(
f:(Y,J')\rightarrow (X,J)
\)
which is homotopic to $h$.
\end{definition}

Siu proved that a compact K\"ahler manifold whose curvature tensor is strongly negative in his sense is
marked strongly rigid \cite[Lemma 2, Theorem 6, 8]{Siurigidity}.  More
precisely, one first chooses a harmonic map in the prescribed homotopy
class; Siu's curvature condition forces this harmonic map to be
holomorphic or anti-holomorphic, and the degree and homological
assumptions imply that it is a biholomorphism or anti-biholomorphism.

Siu's strong negativity is a condition on the complexified curvature
tensor and is stronger than ordinary negativity of the Riemannian
sectional curvature.  However, in complex dimension two, Zheng \cite{Zhengrigidity} established the
corresponding rigidity result under the weaker Riemannian curvature
hypothesis. See also \cite{ChenYang}.

\begin{theorem}[\cite{Siurigidity,Zhengrigidity}]\label{thm:rigidity}
    Any compact K\"ahler surface admitting a
K\"ahler metric of negative sectional curvature is marked strongly
rigid.
\end{theorem}

More generally, Zheng \cite[Proposition~3]{Zhengrigidity} even proved that a
nonpositively curved K\"ahler surface of general type satisfying
\(
c_1^2(X)>2c_2(X)
\)
is marked strongly rigid. An important class of examples is given by compact complex ball
quotients.  Recall that such a surface has the form
\(
X=\Pi\backslash B_{\CC}^2,
\)
where $\Pi<\operatorname{PU}(2,1)$ is a torsion-free cocompact
lattice and $ B_{\CC}^2$ is the complex unit ball.  The standard
complex-hyperbolic metric on $ B_{\CC}^2$ is invariant under
$\operatorname{PU}(2,1)$ and descends to $X$.  Its curvature tensor is
strongly negative even in the sense of Siu, so compact complex ball
quotients are marked strongly rigid by Theorem \ref{thm:rigidity}.  

\begin{corollary}\label{cor:rigidity}
Suppose that $X$ underlies a K\"ahler surface of negative sectional
curvature. Let $J$ and $J'$ be two complex structures on $X$. Then
\(
H_J^{1,1}(X;\RR)=H_{J'}^{1,1}(X;\RR).
\)
Moreover, if
\(
\mathcal K_J^c(X)\cap\mathcal K_{J'}^c(X)\neq\varnothing,
\)
then there exists $f\in\Dfh(X)$ such that
\(
f^*J'=J.
\)
\end{corollary}

\begin{proof}
By Theorem~\ref{thm:rigidity}, $(X,J)$ is marked strongly rigid.
Applying marked strong rigidity to the identity homotopy equivalence
between $(X,J)$ and $(X,J')$, we obtain a diffeomorphism
$f:X\to X$, homotopic to the identity, such that
\(
f^*J'=\pm J\). Since
\(
H_{-J}^{1,1}(X;\RR)=H_J^{1,1}(X;\RR),
\)
we obtain
\[
H_J^{1,1}(X;\RR)
=
H_{f^*J'}^{1,1}(X;\RR)
=
f^*H_{J'}^{1,1}(X;\RR)
=
H_{J'}^{1,1}(X;\RR).
\]
Now suppose that
\(
a\in\mathcal K_J^c(X)\cap\mathcal K_{J'}^c(X).
\)
Choose forms
\(
\omega\in\mathcal{SK}_J(X),
\omega'\in\mathcal{SK}_{J'}(X)
\)
representing the class $a$.
It follows that $f^*J'=-J$ cannot occur. Indeed,
in this case $f^*\omega'$ is K\"ahler with respect to $-J$, and hence
$-f^*\omega'$ is K\"ahler with respect to $J$. Since $f\in\Dfh(X)$, $\mathcal{K}_J^c(X)$ would contain both $a$ and $-a$, which is impossible. Therefore, $f$ must be a biholomorphism.
\end{proof}

\subsection{Moduli of complex structures}
 We now investigate Question~\ref{question:connected} concerning connectedness. The guiding principle is that the number of connected components of the moduli space of complex structures should provide an upper bound for the number of deformation equivalence classes of both holomorphically tamed and K\"ahler-type symplectic forms. 
 
Recall that the Teichm\"uller space of a closed oriented smooth $4$-manifold $X$ is defined to be
\[\text{Teich}(X):=\{\text{complex structures on $X$}\}/\Dff(X).\]
Like $\widetilde{\mathcal{M}}(X)$, $\text{Teich}(X)$ is also well-known to be possibly non-Hausdorff (\cite{superficial}). One can further take its quotient by either Torelli group or mapping class group
\[\widetilde{\mathcal{I}}(X):=\text{Teich}(X)/\Gamma_h(X),\quad \mathcal{I}(X):=\text{Teich}(X)/\Gamma(X).\]
For any $a\in H^2(X;\RR)$, let us also denote by \[\widetilde{\mathcal{I}}_a(X)\] the union of those connected components of $\widetilde{\mathcal{I}}(X)$ that contain a complex structure $J$ satisfying $a\in\mathcal{K}_J^c(X)$. Clearly, $\#\pi_0(\widetilde{\mathcal{I}}_a(X))\leq \#\pi_0(\widetilde{\mathcal{I}}(X))$ and $\#\pi_0(\mathcal{I}(X))\leq \#\pi_0(\widetilde{\mathcal{I}}(X))$. Using these notations, Theorem \ref{thm:rigidity} and Corollary \ref{cor:rigidity} can be equivalently formulated as \[\#\pi_0(\widetilde{\mathcal{I}}(X))\leq 2,\quad\#\pi_0(\widetilde{\mathcal{I}}_a(X))\leq 1\]
if $X$ underlies a K\"ahler surface with negative sectional curvature.

\begin{theorem}\label{thm:connected}
The numbers of connected components of $\mathcal{MT}(X)$ and
$\mathcal{MK}(X)$ are both bounded above by the number of connected
components of $\mathcal{I}(X)$. More precisely,
\[
\#\pi_0(\mathcal{MT}(X))
\leq
\#\pi_0(\mathcal{I}(X)),
\qquad
\#\pi_0(\mathcal{MK}(X))
\leq
\#\pi_0(\mathcal{I}(X)).
\]
\end{theorem}

\begin{proof}
 We may assume $\mathcal{MT}(X)\neq\varnothing$. By Theorem~\ref{thm:LZb1}, it follows that $\mathcal{MK}(X)\neq\varnothing$. As in the proof of Lemma \ref{lemma:uniontamed}, we thus know that every complex structure $J$ on $X$ is K\"ahler by Buchdahl--Lamari criterion. Consequently,
if $C$ is a connected component of $\mathcal I(X)$, by Lemma \ref{lem:isotopic},
\(
\big(\bigcup_{J\in C}\mathcal{ST}_J(X)\big)/\Df(X)
\)
is connected, while Lemma \ref{lemma:uniontamed} implies that
\(
\big(\bigcup_{J\in C}\mathcal{SK}_J(X)\big)/\Df(X)
\)
is connected. Thus, each connected component of $\mathcal I(X)$ contributes at most one connected component to $\mathcal{MT}(X)$ and to $\mathcal{MK}(X)$. Therefore, we have the upper bound.
\end{proof}

It follows from the following theorem of Friedman--Morgan that
$\#\pi_0(\mathcal I(X))$ is always finite whenever $X$ underlies a K\"ahler surface.

\begin{theorem}[{\cite[Theorem~S.2]{FMbook}}]\label{thm:FM main}
Let $X$ be a smooth closed oriented $4$-manifold with
$b_1(X)\neq 1$. Then there are only finitely many deformation-equivalence
classes of complex surfaces orientation-preservingly diffeomorphic to $X$.
\end{theorem}

Therefore, Theorem~\ref{thm:connected} can convert the rigidity results about complex structures in Theorems \ref{thm:rigidity} and \ref{thm:FM main} into the following finiteness result for holomorphically tamed and K\"ahler-type forms.

\begin{corollary}\label{cor:general type deformation finite}
For every smooth closed oriented $4$-manifold $X$, the spaces
$\mathcal{MT}(X)$ and $\mathcal{MK}(X)$ have only finitely many
connected components; if $X$ admits a K\"ahler metric of negative sectional curvature, then they can only have at most two connected components.
\end{corollary}

\section{$4$-manifolds with $b_2^+=1$}\label{section:b+1}
In this section, we focus on symplectic $4$-manifolds $(X,\omega)$ with
$b_2^+(X)=1$. Recall that a \emph{deformation} between two symplectic
forms $\omega_0$ and $\omega_1$ is a path
$\{\omega_t\}_{t\in[0,1]}$ of symplectic forms connecting them. Such a
deformation is called an \emph{isotopy} if the cohomology class
\(
[\omega_t]\in H^2(X;\RR)
\)
is independent of $t$.

The manifold $X$ is said to be of \emph{Seiberg--Witten simple type} if
its Gromov--Taubes invariants vanish for every class
$A\in H_2(X;\ZZ)$ whose Gromov--Taubes moduli space has positive
expected dimension. Equivalently, every class $A$ with nonzero
Gromov--Taubes invariant satisfies
\(
A^2+c_1(\omega)\cdot A=0.
\)
When $X$ is not of Seiberg--Witten simple type, one obtains enough positive classes represented by embedded symplectic
surfaces, providing the classes needed for the inflation argument.
The following deformation-to-isotopy result was first proved in
\cite[Theorem~1.2]{deftoiso} using a parametrized version of the
inflation lemma. It was later observed in
\cite[Proposition~4.11]{LiLiuJDG} that the same conclusion holds for
all symplectic $4$-manifolds with $b_2^+(X)=1$.

\begin{theorem}[Deformation to isotopy \cite{deftoiso,LiLiuJDG}]\label{thm:deftoiso}
    Let $(X,\omega)$ be a symplectic 4-manifold with $b_2^+(X)=1$. Then any deformation between two cohomologous symplectic forms on X
may be homotoped through deformations with fixed endpoints to an isotopy.
\end{theorem}

Using the preceding result, we can now provide answers to questions \ref{question:sympST}, \ref{question:open} and \ref{question:unique} in the case $b_2^+=1$ by a straightforward argument.

\begin{theorem}\label{thm:b+=1 main}
    Let $X$ be the underlying smooth closed $4$-manifold of a (not necessarily minimal) K\"ahler surface with $b_2^+(X)=1$. Then, we have the following. 
    \begin{enumerate}
    \item $\mathcal{K}_{\text{Int}}^t(X)=\mathcal{K}_{\text{Int}}^c(X)$ and $\mathcal{ST}(X)=\mathcal{SK}(X).$
    \item $\mathcal{SK}(X)\hookrightarrow\mathcal{S}(X)$ is open.
    \item \(\#\mathcal{MK}_a(X)\leq \#\pi_0(\widetilde{\mathcal{I}}_a(X))\) for any $a\in\mathcal{K}_{\text{Int}}^c(X).$
    \end{enumerate}
\end{theorem}

\begin{proof}
    (1) By Theorem \ref{thm:LZb1}, $\mathcal{K}_J^t(X)=\mathcal{K}_J^c(X)$ for any integrable $J$. Taking the union over all complex structures, one has $\mathcal{K}_{\text{Int}}^t(X)=\mathcal{K}_{\text{Int}}^c(X).$
    To see $\mathcal{ST}(X)=\mathcal{SK}(X)$, it suffices to prove $\mathcal{ST}(X)\subseteq\mathcal{SK}(X)$. Let $\omega$ be a holomorphically tamed symplectic form, and suppose that
\(
\omega\in\mathcal{ST}_J(X)
\)
for some complex structure $J$. By the first item of Theorem~\ref{thm:LZb1}, the $J$-compatible cone $\mathcal K_J^c(X)$ is nonempty. Moreover, the second item implies that there exists
\(
\omega'\in\mathcal{SK}_J(X)
\)
such that
\(
[\omega']=[\omega].
\)
Consider the linear interpolation
\(
\omega_t:=(1-t)\omega+t\omega'.
\)
Since both $\omega$ and $\omega'$ tame $J$, every $\omega_t$ also tames $J$. Thus, $\{\omega_t\}$ is an isotopy from $\omega$ to $\omega'$. By Moser's stability, there exists a diffeomorphism
\(
\phi\in\Dff(X)
\)
such that
\(
\phi^*\omega'=\omega.
\)
Since $\omega'$ is K\"ahler with respect to $J$, the form $\omega$ is K\"ahler with respect to $\phi^*J$. Therefore,
\(
\omega\in\mathcal{SK}_{\phi^*J}(X)\subseteq\mathcal{SK}(X).
\)

(2) Since tameness with respect to a fixed complex structure $J$ is an open condition, the inclusion
\(
\mathcal{ST}(X)\hookrightarrow\mathcal{S}(X)
\)
is open. By item~(1), we have
\(
\mathcal{ST}(X)=\mathcal{SK}(X).
\)
Therefore, the inclusion
\(
\mathcal{SK}(X)\hookrightarrow\mathcal{S}(X)
\)
is also open.

(3) Let
\(
\omega\in\mathcal{SK}_J(X), 
\omega'\in\mathcal{SK}_{J'}(X)
\)
be two K\"ahler-type symplectic forms representing the class
\(
[\omega]=[\omega']=a.
\)
By the definitions of $\widetilde{\mathcal I}_a(X)$ and $\mathcal{MK}_a(X)$, it suffices to consider the case in which $J$ and $J'$ are connected by a path of complex structures and to show that there exists
\(
\phi\in\Dfh(X)
\)
such that
\(
\phi^*\omega'=\omega.
\) 

Since $X$ admits a K\"ahler structure $(\omega,J)$, $b_1(X)$ is even. By the Buchdahl--Lamari criterion, every complex structure along the path is therefore K\"ahler, and hence its integrable tamed cone is nonempty. We may thus apply Lemma \ref{lem:isotopic} or \ref{lemma:uniontamed} to connect $\omega$ and $\omega'$ by a path of symplectic forms. Since $b_2^+(X)=1$ and the endpoints are cohomologous, Theorem~\ref{thm:deftoiso} implies that this deformation can be homotoped, with fixed endpoints, to an isotopy. Moser's stability then yields a diffeomorphism $\phi$ isotopic to the identity such that
\(
\phi^*\omega'=\omega.
\)
In particular, $\phi$ acts trivially on homology, so $\phi\in\Dfh(X)$.

\end{proof}

\begin{rmk}\label{rmk: alt proof of connectedness}
Item~(1) above also yields an alternative proof of the statement in Theorem~\ref{thm:connected} concerning K\"ahler-type forms, without invoking the Kodaira--Spencer stability theorem in the proof of Lemma \ref{lemma:uniontamed}. Indeed, Lemma \ref{lem:isotopic} still gives
\(
\#\pi_0(\mathcal{MT}(X))
\leq
\#\pi_0(\mathcal{I}(X)),
\)
and the equality
\(
\mathcal{MT}(X)=\mathcal{MK}(X)
\)
then implies the corresponding inequality for $\mathcal{MK}(X)$.
\end{rmk}

\begin{rmk}
Item~(2) above admits the following higher-dimensional generalization that appears in the discussion following
\cite[Theorem~2]{FPPZ}, where it is attributed to Professor Dietmar Salamon: for any closed smooth manifold $X$, the inclusion
\[
\bigcup_{h_J^{2,0}=0}
\mathcal{SK}_J(X)
\hookrightarrow
\mathcal S(X)
\]
is open.
Indeed, if $h_J^{2,0}=0$, then the K\"ahler cone $\mathcal K_J^c(X)$ is open in
$H^2(X;\RR)$. Let $\omega\in\mathcal{SK}_J(X)$. By the openness of the
$J$-taming condition and the continuity of the cohomology class map in
Lemma~\ref{lem: continuity of the class map}, there exists a
neighborhood
\(
\mathcal U\subseteq\mathcal S(X)
\)
of $\omega$ such that every $\omega'\in\mathcal U$ tames $J$ and
satisfies
\(
[\omega']\in\mathcal K_J^c(X).
\)
Thus, for every $\omega'\in\mathcal U$, there exists a cohomologous $J$-K\"ahler
form $\omega''$.
Since both $\omega'$ and $\omega''$ tame $J$, their linear interpolation yields a Moser isotopy.
Thus, $\omega'$ is K\"ahler with respect to $f^*J$ for some $f\in\Dff(X)$ and hence $\mathcal{U}\subseteq \mathcal{SK}(X)$.
\end{rmk}

\begin{rmk}
    The result by Catanese in \cite{Catanesecanonical} suggests that Manetti surfaces \cite{Manetti} might provide examples in which the inequality  \(\#\mathcal{MK}_a(X)\leq \#\pi_0(\widetilde{\mathcal{I}}_a(X))\) in item (3) is strict, where $a$ is taken to be the canonical class $K_X$. 
\end{rmk}

We now apply Theorem \ref{thm:b+=1 main} to derive more explicit results for certain $4$-manifolds with $b_2^+=1$ whose moduli spaces of complex structures are well understood.

For Calabi--Yau K\"ahler surfaces $X$ with $b_2^+(X)=1$, classification of complex surfaces tells us that $X$ must be either an Enriques surface or a hyperelliptic surface. Enriques surfaces have the unique diffeomorphism type which can be obtained from the quotient of K3 surfaces by a free $\ZZ_2$-involution \cite{Enriquesbook}. Hyperelliptic surfaces are classified by Bagnera-de Franchis \cite{BdF1908} and Enriques-Severi \cite{EnriquesSeveri} to be quotients $(E\times F)/G'$ of the product between elliptic curves by a finite group \[G'\in\{\ZZ_2,\ZZ_2\oplus\ZZ_2,\ZZ_3,\ZZ_3\oplus\ZZ_3,\ZZ_4,\ZZ_2\oplus\ZZ_4,\ZZ_6\}.\] Suwa \cite{Suwahyperelliptic} later showed that these give rise to seven diffeomorphism types and described them as the quotient $A/G$ of the Abelian surface $A=(E\times F)/H$ by $G=G'/H$, where $H$ denotes the subgroup of $G$ that acts by translations on $E\times F$. $H$ is trivial when $G'=\ZZ_2,\ZZ_3,\ZZ_4,\ZZ_6$ and equal to the $\ZZ_2,\ZZ_3,\ZZ_2$-summand when $G'=\ZZ_2\oplus\ZZ_2,\ZZ_3\oplus\ZZ_3,\ZZ_2\oplus\ZZ_4$ respectively.
Let $\Lambda:=\pi_1(A)\cong H_1(A;\ZZ)$. To each diffeomorphism type $X=A/G$, one can associate an exact sequence 
\[1\rightarrow \Lambda\rightarrow\Gamma\rightarrow G\rightarrow 1,\]
where $\Gamma:=\pi_1(X)$ is called an abstract Euclidean cristallographic group in \cite[Definition 2]{Catanesehyperelliptic}.

\begin{corollary}\label{cor:Enriques hyperellpitic uniqueness}
    Let $X$ be the underlying smooth oriented $4$-manifold of an Enriques surface, a hyperelliptic surface or their blowups. Then
    \(
        \#\mathcal{MK}_a(X)\leq 1
    \)
    for every $a\in H^2(X;\RR)$. In other words, any two cohomologous
    symplectic forms of K\"ahler type on $X$ are related by a homologically trivial diffeomorphism.
\end{corollary}

\begin{proof}
     By Theorem \ref{thm:b+=1 main} (3), it suffices to show that $\#\pi_0(\widetilde{\mathcal{I}}_a(X))\leq 1$. We first consider the Enriques manifold $X$. For any two complex structures $J_1,J_2$ on $X$ such that $a\in \mathcal{K}_{J_1}^c(X)\cap\mathcal{K}_{J_2}^c(X)$, since the underlying manifold $X$ is fixed, we have a canonical identification of the markings on two Enriques surfaces $(X,J_1)$ and $(X,J_2)$. By the connectedness of the period domain proved in \cite[Corollary 1.16]{Namikawa}, we may assume that $J_1,J_2$ can be deformed to some $J_1',J_2'$ with the same period. By the global Torelli theorem of marked Enriques surfaces \cite[Theorem 4.10, Remark 4.11]{Namikawa}, since the identity map on $H^2(X;\ZZ)$ preserves both the period and the K\"ahler class, it can be realized as a homologically trivial biholomorphism from $(X,J'_1)$ to $(X,J'_2)$. It follows that $\widetilde{\mathcal{I}}_a(X)$ is connected.
    
    We next consider a hyperelliptic manifold $X=A/G$ where $G$ is
    cyclic of order $n\in\{2,3,4,6\}$. For any complex structure $J$ on $X$, let
    \(
        (A,\tilde{J})\rightarrow (X,J)
    \)
    be its Abelian surface cover. By identifying $(A,\tilde{J})$ with its Albanese torus $H^0(A;\Omega_A^1)^*/H_1(A;\ZZ)$, we may assume the complex structure $\tilde{J}$ descends from a $G$-invariant linear complex structure on the vector space $V:=\Lambda\otimes_\ZZ\RR$. Its real representation by $G$
    admits a canonical decomposition
    \(
       V=U\oplus W,
    \)
    where $U$ is the two-dimensional trivial representation and $W$
    is the nontrivial two-dimensional representation. A $G$-invariant linear complex structure $I$ on $V$ preserves the above
    decomposition and has the form
    \(
        I=I_U\oplus I_W.
    \)
    If $n=2$, both $I_U$ and $I_W$ vary through the space of complex
    structures on a real two-plane. If $n=3,4$, or $6$, the complex
    structure on $W$ is determined up to sign, while $I_U$ remains
    arbitrary. In either case, after fixing the orientation of $X$ and requiring the fixed class $a$ to be contained in the K\"ahler cone, the space of $G$-invariant linear complex structures on $V$ is connected: in Suwa's description it is
    parameterized by $\mathbb H^2$ for $G'=\ZZ_2,\ZZ_2\oplus\ZZ_2$ and by
    $\mathbb H$ for the remaining five types, where $\mathbb H=\{z\in \CC\,|\,\text{Im } z>0\}$; see
    \cite{Suwahyperelliptic,Tsuchihashi,Boa21}. This also follows from
    the description of the Teichm\"uller space as the $G$-fixed locus of
    $\text{Teich}(\mathbb T^4)$ in \cite[Theorem 1]{Catanesehyperelliptic}. This implies that at most one
    $\Dfh(X)$-equivalence class of deformation components of complex structures is relevant
    to a fixed class $a$. Therefore, for hyperelliptic surfaces, we also have the connectedness of $\widetilde{\mathcal{I}}_a(X)$.

The connectedness statement for the minimal surfaces extends to their
blowups. Indeed, suppose that $\widetilde{\mathcal I}_a(X)$ is connected,
and let
\(
X_k=X\#k\overline{\CC\PP}^{2}.
\) 
Every complex structure on $X_k$ is obtained by successively blowing up
a complex structure on $X$. Take some $\tilde{a}\in H^2(X_k;\RR)$ whose projection to $H^2(X;\RR)$ is $a$. Given two complex structures $I_0$ and
$I_1$ on $X_k$ such that $\tilde{a}\in\mathcal{K}_{I_0}^c(X_k)\cap \mathcal{K}_{I_1}^c(X_k)$, their minimal models can be connected by a path of
complex structures after composing with a
homologically trivial diffeomorphism of $X$. The associated blowup deformation family (see Section \ref{section:cpx str on blowup}),
together with the connectedness of the space of choices of blowup
points (Lemma \ref{lem: moving blowup point}), then gives a path connecting the corresponding blowup complex
structures. Consequently,
\(
\widetilde{\mathcal I}_{\tilde{a}}(X_k)
\)
is also connected.

\end{proof}

For surfaces of general type with $b_2^+(X)=1$, there is still a finiteness result for K\"ahler-type forms in a fixed class, parallel to Corollary \ref{cor:general type deformation finite}.

\begin{corollary}\label{cor:general type uniqueness}
    Let $X$ be the underlying smooth $4$-manifold of a (not necessarily minimal) K\"ahler surface of general type with $b_2^+(X)=1$. Then $\#\mathcal{MK}_a(X)$ is finite for any $a\in H^2(X;\RR)$.
\end{corollary}

\begin{proof}
  Note that the group $\widetilde{\mathcal{I}}(X)/\mathcal{I}(X)\cong \Df(X)/\Dfh(X)$ can be identified with the image of $\Df(X)\rightarrow \text{Aut}(H^2(X;\ZZ))$. This group is known to be finite by \cite[Corollary 4.8]{FM97}: any diffeomorphism $f$ has to preserve the canonical class and exceptional classes up to signs by Seiberg--Witten theory, and the automorphism of their orthogonal complement, which is a negative definite space, is finite. Together with Theorem \ref{thm:FM main}, this implies $\#\pi_0(\widetilde{\mathcal{I}}(X))<\infty$ and hence  $\#\pi_0(\widetilde{\mathcal{I}}_a(X))<\infty$. Thus, the finiteness of $\#\mathcal{MK}_a(X)$ follows from Theorem \ref{thm:b+=1 main} (3).
\end{proof}

\section{$4$-manifolds with $b_2^+>1$: $K3$, $\mathbb{T}^4$, and their blowups}\label{section:k3}

We now turn to the more delicate $b_2^+>1$ case. In the absence of the `deformation-to-isotopy' Theorem~\ref{thm:deftoiso}, we must extract more refined information from the moduli spaces of complex structures.

\subsection{Period domain and Torelli Theorem}\label{section:torelli}

Let us begin by recalling some standard results from the Torelli theory of K3 surfaces and complex tori. We refer to \cite{BHPVbook,BurnsRapoport,Todorov,Siu,LooijengaK3,K3book,Shioda} for further details.

 Throughout this section, $X$ denotes the underlying smooth manifold of either a K3 surface or the four-torus $\mathbb T^4$. Denote by $(L,\langle\,,\,\rangle)$ either the cohomology lattice $L_{K3}=3H\oplus 2 E_8$ or $L_{\mathbb T^4}=3H$, where $H$ is the rank $2$ hyperbolic plane and $E_8$ is the negative-definite even unimodular lattice of rank $8$ associated to the $E_8$ root system. Let $L_{\mathbb R}, L_{\CC}$ be $L\otimes\RR=H^2(X;\RR),L\otimes \CC=H^2(X;\CC)$ respectively.

The period domain is defined to be 
$$\Phi=\{\varphi\in \mathbb P(L_{\mathbb C})\,|\,\langle \varphi, \varphi\rangle =0, \langle \varphi, \bar \varphi\rangle >0\},$$ 
which is an open subset of the non-singular quadric $\{\langle \varphi,\varphi\rangle=0\}$ in $\mathbb {CP}^{21}$ when $X=K3$; or $\mathbb{CP}^5$ when $X=\mathbb T^4$, hence a complex manifold of complex dimension $20$ or $4$ respectively.  Note that a point $\varphi\in \Phi$ determines a Hodge structure on $L_\mathbb C$:
$$H^{2, 0}=\mathbb C \varphi, \quad H^{0, 2}=\mathbb C \bar \varphi, \quad H^{1, 1}=(H^{2, 0}\oplus H^{0, 2})^\perp.$$
The period map is defined by
\[\textup{Per}:\{\text{complex structures on }X\}\rightarrow \Phi,\quad J\mapsto \text{Per}(J)\] 
where $\textup{Per}(J) $ is the point in $\Phi$ representing the complex line $H^{2, 0}_J(X)\subseteq L_\CC$. It can be considered as a surjective map from the space of marked $K3$ or $\mathbb T^4$, with the obvious marking on $(X,J)$ since we have fixed the underlying manifold $X$ and its cohomology lattice $L$, to the period domain $\Phi$. Hence $\Phi$ is a ``universal parameter space'' of all complex structures, although different points in the period domain may represent biholomorphic complex structures due to the  action by diffeomorphism group.  If $J,J'$ are two complex structures on $X$ such that $\textup{Per}(J)=\textup{Per}(J')$, then there exists some $f\in \textup{Diff}_+(X)$ which is a biholomorphism between $(X,J)$ and $(X,J')$. Note that $f$ may not belong to $\Dfh(X)$, as $\mathcal{K}_J^c(X)$ and $\mathcal{K}_{J'}^c(X)$ may differ by the action of Weyl group but $f^*$ needs to preserve them in the requirement of Torelli theorem.  If \{$\varphi(t)\}$ is a path in the period domain $\Phi$, using local Torelli theorem and a finite subdivision of the interval, one can choose compatible local lifts of the period path and thereby obtain a family of complex structures $\{J(t)\}$ realizing \(\{\varphi(t)\}\). Moreover, the initial $J(0)$ can be chosen to be any complex structure in $\textup{Per}^{-1}(\varphi(0))$.

There is an alternative description of the period domain $\Phi$ using Grassmannians. For a real vector space $V$ with a quadratic form, denote by $G_k^{+,\text{or}}(V)$ the set of all $k$-dimensional positive-definite linear subspaces of $V$ equipped with an orientation. Consider the ``plane'' map
\[P:\Phi\rightarrow G_2^{+,\text{or}}(L_\RR),\qquad \varphi\mapsto P(\varphi)\]
where $P(\varphi)$ is spanned by the ordered basis $\{\text{Re}(\theta),\text{Im}(\theta)\}\subseteq L_\RR$ for any $\theta\in L_\CC$ with $[\theta]=\varphi\in \Phi$. The Hodge-Riemann bilinear relations
 $\langle \varphi, \varphi\rangle =0, \langle \varphi, \bar \varphi\rangle >0$ in the definition of $\Phi$ guarantee that $P$ is well-defined and it is straightforward to verify $P$ is a bijection by constructing its inverse $P^{-1}(W)=[\alpha+i\beta]\in\Phi$, where $\{\alpha,\beta\}$ is an orthonormal basis of the two-plane $W\subseteq L_\RR$ whose order coincides with the given orientation. Henceforth, we will use these two equivalent descriptions of the period domain interchangeably.

Next, for any nonzero class $\kappa\in H^2(X;\RR)$, consider the codimension-one complex submanifold
\[\overline{M}_\kappa:=\{\varphi \in \Phi\,|\,\langle\varphi,\kappa\rangle=0\}\subseteq \Phi.\] 
  In terms of the Grassmannian description,
$\overline{M}_\kappa$ can be identified with $G_2^{+,\text{or}}(\kappa^{\perp})$. When $\kappa^2>0$, $\kappa^{\perp}$ has signature $(2,19)$ for K3; or $(2,3)$ for $\mathbb T^4$, and $G_2^{+,\text{or}}(\kappa^{\perp})$ can be further identified with the homogeneous space $SO(2,19)/\big(SO(2)\times SO(19)\big)$ or $SO(2,3)/\big(SO(2)\times SO(3)\big)$. Note that $SO(2,19)$ and $SO(2,3)$ have two connected components depending on whether the orientation on a maximal positive-definite subspace is preserved or not, while the stabilizers $SO(2)\times SO(19)$ and $SO(2)\times SO(3)$ are connected. Hence, $\overline{M}_\kappa$ has two connected components if $\kappa^2>0$.

Let  $\Delta_\kappa$ be the subset in $L$ defined by
$$\Delta_\kappa:=
    \{\dd\in H^2(X;\ZZ)\,|\,\dd^2=-2, \langle\dd,\kappa\rangle=0\}.$$ 
One can then consider 
\[M_\kappa:=\begin{cases}
    \overline{M}_\kappa\setminus \bigcup_{\dd\in\Delta_\kappa} \overline{M}_\dd,\, X=K3;\\
    \overline{M}_\kappa,\qquad \qquad \quad \,\,\,\,X=\mathbb T^4.
\end{cases}\] 

\begin{lemma} \label{lemma: connected complement}
Let $M$ be a connected smooth manifold, and let
\(
A=\bigcup_{i=1}^{\infty} Z_i
\)
be a countable union of smooth submanifolds $Z_i\subseteq M$ of real codimension at least $2$. Then $M\setminus A$ is path-connected.
\end{lemma}

\begin{proof}
Let $p,q\in M\setminus A$. 
Consider the path space
\[
\mathcal P_{p,q}=\{\gamma\in C^\infty([0,1],M)\,|\, \gamma(0)=p,\ \gamma(1)=q\}
\]
with the $C^\infty$-topology. Since $M$ is connected, $\mathcal{P}_{p,q}$ is a nonempty Fr\'echet manifold. For each $i$, let
\[
\mathcal U_i=\{\gamma\in \mathcal P_{p,q}\,|\, \gamma \text{ is transverse to } Z_i\}.
\]
By the standard transversality theorem, since the endpoints $p,q$ do not lie in $Z_i$, the set $\mathcal U_i$ is an open dense subset in $\mathcal P_{p,q}$.
Therefore, the countable intersection
\(
\bigcap_{i=1}^{\infty} \mathcal U_i
\)
is residual, hence nonempty, by the Baire category theorem. Choose
\(
\gamma\in \bigcap_{i=1}^{\infty} \mathcal U_i.
\)
Then $\gamma$ is transverse to every $Z_i$ and must be disjoint from every $Z_i$ by the codimension assumption. 
Thus
\(
\gamma([0,1])\subseteq M\setminus \bigcup_{i=1}^{\infty} Z_i.
\)
So $p$ and $q$ can be joined by a path inside the complement. Since $p,q\in M\setminus A$ were arbitrary, $M\setminus A$ is path-connected.
\end{proof}

When $\kappa^2>0$, the inclusion $\overline{M}_\kappa\cap\overline{M}_\dd\subseteq\overline{M}_\kappa$ can be viewed as \(G_2^{+,\text{or}}(\kappa^\perp\cap\dd^\perp)\subseteq G_2^{+,\text{or}}(\kappa^\perp)\). Since $\kappa^\perp\cap\dd^\perp$ has signature $(2,18)$ or $(2,2)$, we see that $\dim_\RR\overline{M}_\kappa\cap\overline{M}_\dd=36 \text{ or }4$; while $\dim_\RR\overline{M}_\kappa=38\text{ or }6$. Hence, this is a real codimension-two embedding. It then follows from Lemma \ref{lemma: connected complement} that $M_\kappa$ also has two connected components since only countably many real codimension-two submanifolds are removed from $\overline{M}_\kappa$ in the K3 case, and $\overline{M}_\kappa=M_\kappa$ in the $\mathbb T^4$ case.

Given any $\varphi\in \Phi$, one can define the projection map
\[Z_\varphi:L_\RR\rightarrow P(\varphi)^\perp \subseteq L_\RR,\qquad \kappa\mapsto Z_\varphi(\kappa)\]
 by using the intersection pairing on $L_\RR$. So $Z_{\varphi}(\kappa)$ is the $(1,1)$-part $\kappa_J^{1,1}\in H_J^{1,1}(X;\RR)$ of $\kappa$ with respect to any complex structure $J$ with $\text{Per}(J)=\varphi$. Then, for any $\kappa\in L_\RR$ with $\kappa^2>0$, we consider
\begin{align*}
    \Psi_\kappa^{>0}&:=\{\varphi \in \Phi\,|\, \langle Z_{\varphi}(\kappa), Z_{\varphi}(\kappa)\rangle>0\},\\
    \Psi_\kappa^{=0}&:=\{\varphi \in \Phi\,|\, \langle Z_{\varphi}(\kappa), Z_{\varphi}(\kappa)\rangle=0\},\\
    \Psi_\kappa^{*}&:=\{\varphi \in \Phi\,|\,  Z_{\varphi}(\kappa)=0\},\\
    \Psi_\kappa^{<0}&:=\{\varphi \in \Phi\,|\, \langle Z_{\varphi}(\kappa), Z_{\varphi}(\kappa)\rangle<0\}, \\
    \Phi_\kappa&:=\begin{cases}
    \Psi_\kappa^{>0}\setminus\bigcup_{\dd\in\Delta_\kappa}\overline{M}_\dd,X=K3;\\
    \Psi_\kappa^{>0},\qquad\qquad\,\,\,\,\,\,\,\,\, X=\mathbb T^4.
\end{cases}
\end{align*}

The following lemma records the relevant topological properties of these subdomains of $\Phi$. A schematic illustration is given in Figure~\ref{fig:period domain}.

\begin{lemma}\label{lem:Grassmannian}
Let $\kappa\in L_\RR$ be a class with $\kappa^2>0$. 
\begin{enumerate}
    \item $\Psi^{<0}_\kappa$ is a connected open subdomain of $\Phi$.
     \item $\Psi^{=0}_{\kappa}$ is a connected singular wall between $\Psi^{>0}_\kappa$ and $\Psi^{<0}_\kappa$, with $\Psi^*_\kappa$ being its singularity set.
    \item $\Psi_\kappa^{>0}$ (resp. $\Phi_\kappa$) is an open two-disk bundle over $\overline{M}_\kappa$ (resp. $M_\kappa$). In particular, $\Psi_\kappa^{>0}$ and $\Phi_\kappa$ are subdomains of $\Phi$ with exactly two connected components.
     \item For any $\varphi,\overline{\varphi}\in \overline{M}_\kappa$, let $\mathbb D_\varphi,\mathbb D_{\overline{\varphi}}$ be the disk fibers of the fibration $\Psi_\kappa^{>0}\rightarrow\overline{M}_{\kappa}$. Then $\overline{\mathbb D}_\varphi\setminus \mathbb D_\varphi=\overline{\mathbb D}_{\overline{\varphi}}\setminus \mathbb D_{\overline{\varphi}}$ is an embedded circle in $\Psi_{\kappa}^*$. Moreover, $\overline{\mathbb D}_\varphi\cup \overline{\mathbb D}_{\overline{\varphi}}$ corresponds to the periods of a hyperK\"ahler $S^2$-family.
\end{enumerate}
 
\end{lemma}

\begin{proof}
We may assume $\kappa^2=1$ and use $a^2$ to denote the intersection pairing $\langle a,a\rangle$ for any real class $a\in L_\RR$ in the proof below for convenience.

To prove the first two items, observe that for $\varphi\in \Psi_\kappa^{<0}\cup \Psi_\kappa^{=0}$, the vector $\kappa-Z_\varphi(\kappa)\in P(\varphi)$ satisfies $$\langle  \kappa-Z_\varphi(\kappa),\kappa\rangle=1-Z_\varphi(\kappa)^2> 0.$$ So, there exists a unique vector $x(\varphi)\in P(\varphi)\cap \kappa^\perp$ such that $x(\varphi)^2=1$ and the ordered basis $\{x(\varphi),\kappa-Z_\varphi(\kappa)\}$ gives the orientation of $P(\varphi)$. Define $y(\varphi):= Z_{\varphi}(\kappa)^2\kappa-Z_\varphi(\kappa)$. One can verify $y(\varphi)\in \kappa^\perp\cap x(\varphi)^\perp$.

(1) When $Z_{\varphi}(\kappa)^2<0$, notice that $y(\varphi)^2=- Z_{\varphi}(\kappa)^2\cdot Z_{\varphi}(\kappa)^2+Z_{\varphi}(\kappa)^2<0$. Therefore, one can consider the map
\[\Psi_\kappa^{<0}\rightarrow \Theta^{<0}:=\{(x,y)\in\kappa^\perp\times\kappa^\perp\,|\,x^2=1,y^2<0,\langle x,y\rangle=0\},\,\, \varphi\mapsto(x(\varphi),y(\varphi)). \]
This is a bijection, as it has an inverse map by sending any $(x,y)\in \Theta^{<0}$ to the oriented positive-definite two-plane spanned by the ordered basis $\{x,y+\frac{1+\sqrt{1-4y^2}}{2}\kappa\}$. Hence, $\Psi_\kappa^{<0}$ can be identified with $\Theta$. Since $\kappa^\perp$ has signature either $(2,19)$ or $(2,3)$, the space $\Theta^{<0}$ is either an $(S^{18}\times \RR^2)$-bundle over $S^1\times \RR^{19}$ or $(S^{2}\times \RR^2)$-bundle over $S^1\times \RR^{3}$ by projection into the $x$-factor. Therefore, $\Psi^{<0}_\kappa$ is open and connected. 

(2) When $Z_{\varphi}(\kappa)^2=0$, we have $y(\varphi)^2=0$. Similar to the previous case, one can consider the maps
\begin{align*}
    \Psi_\kappa^{=0}\rightarrow \Theta^{=0}&:=\{(x,y)\in\kappa^\perp\times\kappa^\perp\,|\,x^2=1,y^2=0,\langle x,y\rangle=0\},\,\, \varphi\mapsto(x(\varphi),y(\varphi)), \\
    \Psi_\kappa^{*}\rightarrow \Theta^{*}&:=\{(x,y)\in\kappa^\perp\times\kappa^\perp\,|\,x^2=1,y=0\},\,\, \varphi\mapsto(x(\varphi),0). 
\end{align*}
and identify $\Psi^*_\kappa\subseteq \Psi^{=0}_\kappa$ with $\Theta^*\subseteq \Theta^{=0}$. Again, by projecting $\Theta^{=0}$ into the $x$-factor, we see that each fiber is connected with exactly one singularity at $y=0$, which corresponds to the point inside $\Theta^*$.

(3)  When $Z_\varphi(\kappa)^2>0$, $P(\varphi)$ and $\kappa$ span a $3$-dimensional positive-definite subspace $W\subseteq L_\RR$. Hence, we can define the map
\(\pi:\Psi_\kappa^{>0}\rightarrow \overline{M}_\kappa\) by sending the oriented $2$-plane $P(\varphi)$ to its orthogonal projection oriented $2$-plane inside $\kappa^\perp\cap W$ for any $\varphi\in \Psi_\kappa^{>0}$. Given any $\phi\in \overline{M}_\kappa$, the ordered pair of $\kappa$ and the oriented $P(\phi)$ gives an orientation of $W$. One can identify all oriented $2$-planes in $W$ with the set of unit vectors by taking the normal vectors compatible with the orientation of $W$. Therefore, the fiber $\pi^{-1}(\phi)$ is an open hemisphere so that $\pi$ is an open $2$-disk fibration.

In the K3 case, suppose that
\(
\varphi\in\overline{M}_\kappa\cap\overline{M}_\dd
\)
for some $\dd\in\Delta_\kappa$. Then the associated positive $3$-plane $W$ is orthogonal to $\dd$. Consequently, the entire $2$-disk fiber $\pi^{-1}(\varphi)$ is contained in $\overline{M}_\dd$. It follows that $\Phi_\kappa$ is obtained by restricting the $2$-disk fibration
\(
\pi
\)
to $M_\kappa\subseteq\overline{M}_\kappa$.

Since $\overline{M}_\kappa$ has two connected components, so does $\Psi_\kappa^{>0}$. By Lemma \ref{lemma: connected complement}, $\Phi_\kappa$ also has two connected components.

(4) Notice that $P(\varphi)$ and $P(\overline{\varphi})$ are the same $2$-plane in $L_\RR$ equipped with the opposite orientations. The disk fibers $\mathbb D_\varphi,\mathbb D_{\overline{\varphi}}$ can both be parametrized by the points on the unit sphere inside the $3$-dimensional subspace $W$ considered in (3). Since $P(\varphi)$ and $P(\overline{\varphi})$ have opposite orientations, $\mathbb D_\varphi$ and $\mathbb D_{\overline{\varphi}}$ correspond to two disjoint hemispheres, whose closures have the common boundary circle given by the equator. The vectors on the equator are perpendicular to $\kappa$, so that their corresponding period $2$-planes contain $\kappa$. So, the circle is embedded in $\Psi_\kappa^*$. The whole sphere $\overline{\mathbb D}_\varphi\cup \overline{\mathbb D}_{\overline{\varphi}}$ gives rise to a hyperK\"ahler family since $W$ is positive-definite.

\end{proof}

\begin{figure}[h]\label{fig:period domain}
    \centering

\begin{tikzpicture}[
    scale=0.8,
    every node/.style={font=\large},
    thickline/.style={line width=1.4pt, draw=black!70},
    medline/.style={line width=0.75pt, draw=black!65},
    dashcurve/.style={line width=1.2pt, draw=black!70, densely dotted},
    grayfill/.style={fill=black!30, draw=none}
]

% ------------------------------------------------------------
% Central three half-planes / pages
% ------------------------------------------------------------
\coordinate (O) at (0,-2.45);
\coordinate (L) at (-1.35,1.25);
\coordinate (R) at (1.35,1.25);

% central shaded page
\fill[grayfill] (O) -- (L) -- (R) -- cycle;

% three boundary rays Psi_kappa^{=0}
\draw[thickline] (L) -- (-4.55,5.00);
\draw[thickline] (R) -- (4.55,5.00);
\draw[thickline] (O) -- (0,-4.00);

% NOTE: the top edge (L)--(R) is intentionally not drawn.

% ------------------------------------------------------------
% Two gray M_kappa regions
% ------------------------------------------------------------
\fill[grayfill, rotate around={-7:(-4.05,-0.55)}]
    (-4.05,-0.55) ellipse (1.55 and 2.05);

\fill[grayfill, rotate around={7:(4.05,-0.55)}]
    (4.05,-0.55) ellipse (1.55 and 2.05);

% ------------------------------------------------------------
% Large oval in the middle
% Drawn after the shaded M_kappa regions so that it is fully visible.
% ------------------------------------------------------------
\draw[medline]
    (-4.55,-0.75)
    .. controls (-3.40,1.05) and (3.40,1.05) ..
    (4.55,-0.75)
    .. controls (3.40,-2.05) and (-3.40,-2.05) ..
    (-4.55,-0.75);

% marked points
\fill (-4.55,-0.75) circle (0.09);
\fill (4.55,-0.75) circle (0.09);

% ------------------------------------------------------------
% Dashed small central oval
% ------------------------------------------------------------
\draw[dashcurve]
    (0,-1.78)
    .. controls (-0.28,-1.20) and (-0.28,0.35) ..
    (0,0.58)
    .. controls (0.28,0.35) and (0.28,-1.20) ..
    (0,-1.78);

% ------------------------------------------------------------
% Dashed curves around left oval
% ------------------------------------------------------------
\draw[dashcurve]
    (-6.95,0.15)
    .. controls (-6.45,0.48) and (-5.95,0.53) ..
    (-5.45,0.25)
    .. controls (-4.95,0.02) and (-4.65,-0.02) ..
    (-4.20,0.15)
    .. controls (-3.95,0.25) and (-3.80,0.35) ..
    (-3.60,0.48);

\draw[dashcurve]
    (-5.85,-2.50)
    .. controls (-5.95,-2.05) and (-5.75,-1.52) ..
    (-5.35,-1.23)
    .. controls (-5.05,-1.02) and (-4.82,-0.90) ..
    (-4.55,-0.95);

\draw[dashcurve]
    (-3.15,0.48)
    .. controls (-2.82,0.45) and (-2.77,0.70) ..
    (-2.50,0.98)
    .. controls (-2.32,1.18) and (-2.17,1.40) ..
    (-1.88,1.35);

% ------------------------------------------------------------
% Dashed curves around right oval
% ------------------------------------------------------------
\draw[dashcurve]
    (5.15,2.00)
    .. controls (5.12,1.55) and (5.25,1.28) ..
    (5.55,0.98)
    .. controls (5.80,0.70) and (5.75,0.42) ..
    (4.35,0.28);

\draw[dashcurve]
    (4.95,-0.12)
    .. controls (5.35,0.20) and (5.65,-0.05) ..
    (6.00,-0.28)
    .. controls (6.45,-0.60) and (6.85,-0.70) ..
    (7.18,-1.10);

\draw[dashcurve]
    (3.62,-3.05)
    .. controls (3.55,-2.45) and (3.55,-2.05) ..
    (3.45,-2)
    .. controls (3.35,-1.8) and (3.35,-1.6) ..
    (3.22,-1.5);

% ------------------------------------------------------------
% Labels: Psi regions
% ------------------------------------------------------------
\node[font=\LARGE] at (0.30,3.25) {$\Psi_{\kappa}^{<0}$};

\node[font=\LARGE] at (-1.95,-2.88) {$\Psi_{\kappa}^{>0}$};
\node[font=\LARGE] at (1.95,-2.88) {$\Psi_{\kappa}^{>0}$};

\node at (-1.87,2.30) {$\Psi_{\kappa}^{=0}$};
\node at (1.87,2.30) {$\Psi_{\kappa}^{=0}$};
\node at (0.30,-3.65) {$\Psi_{\kappa}^{=0}$};

\node[font=\LARGE] at (0,0.93) {$\Psi_{\kappa}^{*}$};

% ------------------------------------------------------------
% Labels: D regions
% ------------------------------------------------------------
\node at (-1.95,-0.42) {$\mathbb{D}_{\varphi}$};
\node at (1.95,-0.42) {$\mathbb{D}_{\overline{\varphi}}$};

\node at (0,-0.45) {$\overline{\mathbb{D}}_{\varphi}
    \cap \overline{\mathbb{D}}_{\overline{\varphi}}$};

% ------------------------------------------------------------
% Labels: M_kappa regions and points
% ------------------------------------------------------------
\node[font=\LARGE] at (-4.35,-1.82) {$\overline{M}_{\kappa}$};
\node[font=\LARGE] at (4.35,-1.82) {$\overline{M}_{\kappa}$};

\node at (-4.85,-0.75) {$\varphi$};
\node at (4.85,-0.75) {$\overline{\varphi}$};

% ------------------------------------------------------------
% Boundary component labels
% ------------------------------------------------------------
\node at (-6.55,0.95) {$\overline{M}_{\delta_1}$};
\node at (-2.45,1.68) {$\overline{M}_{\delta_2}$};
\node at (-5.65,-2.70) {$\overline{M}_{\delta_3}$};

\node at (5.55,1.58) {$\overline{M}_{\delta_4}$};
\node at (6.83,-1.18) {$\overline{M}_{\delta_5}$};
\node at (4.05,-3.15) {$\overline{M}_{\delta_6}$};
 
\end{tikzpicture}
 
\caption{A schematic picture of the decomposition of the period domain $\Phi$ with respect to $\kappa$. 
The wall $\Psi_\kappa^{=0}$ contains the locus $\Psi_\kappa^*$, shown as the central triangular region, and separates $\Phi$ into the regions $\Psi_\kappa^{<0}$ and $\Psi_\kappa^{>0}$. 
The central sphere joining the periods $\varphi$ and $\overline{\varphi}$ represents a hyperK\"ahler twistor family, which realizes $\Psi_\kappa^{>0}$ as a disk bundle over $\overline{M}_\kappa$. 
When $X=K3$, $M_\kappa$ and $\Phi_\kappa$ are obtained by deleting the dashed curves corresponding to the loci $\overline{M}_{\delta_i}$.}
\end{figure}

The roles played by $M_\kappa$ and $\Phi_\kappa$ are illustrated by the lemma below.

\begin{lemma}\label{lem:kahler and tame period subdomain}
For any $\kappa\in L_\RR$ with $\kappa^2>0$, we have the following.
\begin{enumerate}
    \item There exists a complex structure $J\in \text{Per}^{-1}(\varphi)$ such that $\kappa\in \pm \mathcal{K}^c_J(X)$ (resp. $\kappa\in  \pm \mathcal{K}^t_J(X)$) if and only if  $\varphi\in M_\kappa$ (resp. $\varphi\in \Phi_\kappa$).
    \item  If $\Delta_\kappa=\varnothing$, then $\kappa\in  \pm\mathcal{K}^c_J(X)$ (resp. $\kappa\in  \pm\mathcal{K}^t_J(X)$) if and only if $\text{Per}(J)\in M_\kappa$ (resp. $\text{Per}(J)\in\Phi_\kappa$).
\end{enumerate}
   
\end{lemma}

\begin{proof}
   (1) If $\kappa\in\pm \mathcal{K}_J^c(X)$, then $\kappa\in H^{1,1}_J(X)$. This implies that
   \(
   \langle\kappa,\text{Per}(J)\rangle=0,
   \)
   so that $\text{Per}(J)\in\overline{M}_\kappa$. When $X=K3$, for any $\dd\in\Delta_\kappa$, either $\dd$ or $-\dd$ is represented by an effective $(-2)$-curve by Riemann--Roch. Since a K\"ahler class pairs positively with every effective curve, we must have
   \(
   \text{Per}(J)\notin \overline{M}_\dd .
   \)
   Hence $\text{Per}(J)\in M_\kappa$.
   Conversely, suppose $\varphi\in M_\kappa$. Then $\kappa\in H^{1,1}_J(X)$ and $\kappa$ is not orthogonal to any $(-2)$-class in $H^{1,1}_J(X)$. Hence $\kappa$ lies in the interior of one of the chambers of the positive cone. By the surjectivity part of the Torelli theorem, there exists a complex structure $J\in \text{Per}^{-1}(\varphi)$ such that either $\mathcal{K}_J^c(X)$ or $-\mathcal{K}_J^c(X)$ is this chamber. Therefore
   \(
   \kappa\in \pm\mathcal{K}_J^c(X).
   \)

   The tame statement follows from Theorem \ref{thm:LZb1}. Indeed, if
   \(
   \kappa\in \pm \mathcal{K}_J^t(X),
   \)
   then $Z_\varphi(\kappa)\in \pm\mathcal{K}_J^c(X)$, where $\varphi=\text{Per}(J)$. In particular,
   \(
   Z_\varphi(\kappa)^2>0,
   \) and 
   \(
   \langle\kappa,\dd\rangle=\langle \kappa-Z_\varphi(\kappa),\dd\rangle\neq 0
   \) for any $\dd\in \Delta_\kappa$.
   So $\varphi\in\Phi_\kappa$. Conversely, if $\varphi\in\Phi_\kappa$, then applying the above K\"ahler version statement to $Z_\varphi(\kappa)$, we can choose some $J\in\text{Per}^{-1}(\varphi)$ such that
   \(
   Z_\varphi(\kappa)\in \pm\mathcal{K}_J^c(X).
   \)
   Theorem \ref{thm:LZb1} then implies
   \(
   \kappa\in \pm\mathcal{K}_J^t(X).
   \)

   (2) If $\Delta_\kappa=\varnothing$, then
   \(
   M_\kappa=\overline{M}_\kappa.
   \)
   Therefore, for any complex structure $J$,
   \(
   \kappa\in H^{1,1}_J(X)
   \)
   is equivalent to
   \(
   \text{Per}(J)\in M_\kappa.
   \)
   Since there are no $(-2)$-walls orthogonal to $\kappa$, the class $\kappa$ lies in the one of the two components of the positive cone whenever $\text{Per}(J)\in M_\kappa$. Thus
   \(
   \kappa\in \pm\mathcal{K}_J^c(X)
   \)
   if and only if
   \(
   \text{Per}(J)\in M_\kappa.
   \)
   The tame version follows in the same way from Theorem \ref{thm:LZb1}, replacing $\kappa$ by its projection $Z_{\text{Per}(J)}(\kappa)$. Hence
   \(
   \kappa\in \pm\mathcal{K}_J^t(X)
   \)
   if and only if
   \(
   \text{Per}(J)\in\Phi_\kappa.
   \)

\end{proof}

The following lemma will be used to discuss the sign ambiguity.

\begin{lemma}\label{lem:plus or minus}
Let $J$ be a complex structure on $X$ with $\text{Per}(J)=\varphi\in \Phi_\kappa^+$, where $\Phi_\kappa^+$ is one of the two connected components of $\Phi_k$. Suppose that $\kappa\in  \mathcal{K}^t_J(X)$. Then, $-\kappa\not\in \mathcal{K}_{J'}^t(X)$ for every $J'\in \text{Per}^{-1}(\Phi_\kappa^+)$.
   
\end{lemma}

\begin{proof}
  Suppose first that $\kappa\in\mathcal{K}_J^c(X)$ and that, for some
   $J'\in\text{Per}^{-1}(\varphi)$, one has
   \(
   -\kappa\in\mathcal{K}_{J'}^c(X).
   \)
   Then the Hodge isometry $-\text{Id}$ sends $\mathcal{K}_J^c(X)$ to
   $\mathcal{K}_{J'}^c(X)$. By the Torelli theorem, this Hodge isometry would be induced by a diffeomorphism of $X$. This contradicts Donaldson's theorem \cite{Donaldsonpoly}, which rules out a diffeomorphism of a K3 surface acting by $-\text{Id}$ on $H^2(K3;\ZZ)$; or the simple fact that no automorphism on $H^1(\mathbb T^4;\ZZ)$ can induce $-\text{Id}$ on $H^2(\mathbb T^4;\ZZ)$. Therefore
   \(
   -\kappa\notin\mathcal{K}_{J'}^c(X)
   \)
   for every $J'\in\text{Per}^{-1}(\varphi)$.

   The tame version follows from Theorem \ref{thm:LZb1}. If
   \(
   \kappa\in \pm\mathcal{K}_J^t(X),
   \)
   then
   \(
   Z_{\text{Per}(J)}(\kappa)\in \pm\mathcal{K}_J^c(X).
   \)
   Since $J'$ has the same period as $J$, the projection
   $Z_{\text{Per}(J)}$ is the same for $J'$ and $J$. If
   \(
   -\kappa\in\mathcal{K}_{J'}^t(X),
   \)
   then
   \(
   -Z_{\text{Per}(J)}(\kappa)\in\mathcal{K}_{J'}^c(X),
   \)
   contradicting the K\"ahler case just proved. Hence
   \(
   -\kappa\notin\mathcal{K}_{J'}^t(X)
   \)
   for every $J'\in\text{Per}^{-1}(\varphi)$.

   For a general $J'\in \text{Per}^{-1}(\Phi_\kappa^+)$, we employ a continuity argument. Choose a smooth path of complex structures $\{J_s\}$ with $J_0=J$ and $J_1\in\text{Per}^{-1}(\text{Per}(J'))$. It suffices to show $-\kappa\notin\mathcal{K}_{J_1}^t(X)$ by the previous conclusion. Consider $r:=\sup\{t\in[0,1]\,|\,\kappa\in\mathcal{K}_{J_s}^t(X)\text{ for all }s\in[0,t]\}$. Since $\kappa\in\mathcal{K}_{J_0}^t(X)$, $r$ is a positive number. We must have $r=1$: otherwise, either $-\kappa\in\mathcal{K}_{J_r}^t(X)$ would imply $-\kappa\in\mathcal{K}_{J_{r'}}^t(X)$ for some $r'\in(0,r)$ sufficiently close to $r$; or $\kappa\in\mathcal{K}_{J_r}^t(X)$ would imply there is $r''\in(r,1)$ such that $\kappa\in\mathcal{K}_{J_{s}}^t(X)$ for any $s<r''$. This contracts either the conclusion in the previous paragraph or the definition of $r$. Hence, $\kappa\in\mathcal{K}_{J_s}^t(X)$ for all $s\in[0,1]$
\end{proof}

Finally, let us address the following implication on the connectedness of $\tilde{\mathcal{I}}(X)$ from Torelli theorem for K3 and $\mathbb T^4$. Recall that the complex structure we consider is always assumed to be compatible with a fixed orientation on $X$. $\mathbb T^4$ has an obvious orientation-reversing self-diffeomorphism so that we can refer to any orientation in the discussion below; while for K3 we refer to the specific orientation\footnote{Indeed, the reverse one with $b_2^+(X)=19$ does not admit any complex structure by Enriques-Kodaira classification and Bogomolov–Miyaoka-Yau inequality.} on $X$ with $b_2^+(X)=3$ below.

\begin{lemma}\label{lem:cpxconn}
    For any complex structures $J_0,J_1$ on $X$, there exists some $f\in \textup{Diff}_h(X)$ such that $J_0$ and $f^*J_1$ can be connected by a path of complex structures. Namely, $\tilde{\mathcal{I}}(X)$ is connected.
\end{lemma}

\begin{proof}
For $X=\mathbb{T}^4$, the connectedness of the period domain $\Phi$ implies that $J_0$ can be deformed to a complex structure $J_0'$ satisfying
\(
\text{Per}(J_0')=\operatorname{Per}(J_1).
\)
Once the orientation of $X$ is fixed, Shioda's theorem \cite{Shioda} asserts that any orientation-preserving isomorphism
\(
f^*:H^1(X;\mathbb Z)\rightarrow H^1(X;\mathbb Z)
\)
whose induced action $\wedge^2f^*$ preserves the period can be realized by a diffeomorphism $f$. Choosing $f^*=\operatorname{Id}$, we obtain a biholomorphism
\(
f:(X,J_1)\to (X,J_0'),
\)
which is homologically trivial. Hence $f\in\Dfh(X)$, and $J_0$ is connected to $f^*J_1$ through the deformation from $J_0$ to $J_0'$.

For $X=K3$, the Torelli theorem \cite{LooijengaK3} additionally requires the Hodge isometry to preserve the K\"ahler cones. We first deform $J_0$ and $J_1$ to complex structures $J_0'$ and $J_1'$ with the same period point $\varphi=\text{Per}(J_0')=\text{Per}(J_1')$, chosen so that
\(
\langle\varphi,\delta\rangle\neq0
\)
for every $\delta\in H^2(X;\mathbb Z)$ with $\delta^2=-2$. Such a period corresponds to a K3 surface containing no $(-2)$-curves. Consequently, the K\"ahler cone coincides with one of the two connected components of the positive cone in
\(
H^{1,1}_{J_0'}(X)=H^{1,1}_{J_1'}(X).
\)
By Lemma \ref{lem:plus or minus}, we must have
\(
\mathcal K_{J_0'}^c(X)=\mathcal K_{J_1'}^c(X).
\)
The Torelli theorem now yields a homologically trivial diffeomorphism $f\in\Dfh(X)$ such that
\(
J_0'=f^*J_1'.
\)
Since $J_i$ is connected to $J_i'$ by construction, it follows that $J_0$ is connected to $f^*J_1$ through a path of complex structures.
\end{proof}

\subsection{Minimal case}
We can now provide answers to Questions \ref{question:sympST}, \ref{question:open}, and \ref{question:unique} for minimal manifolds. The statement is almost identical to Theorem \ref{thm:b+=1 main}.
\begin{theorem}\label{thm:minimal main}
    Let $X$ be the underlying smooth closed $4$-manifold of a symplectic Calabi-Yau surface with $b_2^+(X)=3$. Then, we have the following.
    
     \begin{enumerate}
    \item $\mathcal{K}_{\text{Int}}^t(X)=\mathcal{K}_{\text{Int}}^c(X)$ and  $\mathcal{ST}(X)=\mathcal{SK}(X)$.
    \item $\mathcal{SK}(X)\hookrightarrow\mathcal{S}(X)$ is open.
    \item \(\#\mathcal{MK}_\kappa(X)\leq 1\) for any $\kappa\in H^2(X;\RR)$ with $\kappa^2>0$.
    \end{enumerate} 
   
\end{theorem}

\begin{proof}

When $X$ admits no K\"ahler structure, items (1) follows from Theorem \ref{thm:LZb1}, while items (2) and (3) are immediate. When $X$ admits a K\"ahler structure, the Enriques--Kodaira classification together with the condition $b_2^+(X)=3$ enables us to assume that $X$ is either K3 or $\mathbb T^4$ in the following.

(1) For any $\kappa\in H^2(X;\RR)$ with $\kappa^2>0$, the first item in Lemma \ref{lem:kahler and tame period subdomain} guarantees that for any complex structure $J$ with $\text{Per}(J)\in M_\kappa$, we have $\kappa\in\mathcal{K}_J^c(X)\cup\mathcal{K}_{-J}^c(X)$. Therefore, the integrable compatible cone $\mathcal{K}_{\text{Int}}^c(X)$ is the whole positive cone $\mathcal{P}:=\{\kappa\in L_\RR\,|\,\kappa^2>0\}$. Since $\mathcal{K}_{\text{Int}}^c(X)\subseteq\mathcal{K}_{\text{Int}}^t(X)\subseteq \mathcal{P}$, we have the equality $\mathcal{K}_{\text{Int}}^t(X)=\mathcal{K}_{\text{Int}}^c(X)$.

 To see $\mathcal{ST}(X)=\mathcal{SK}(X)$, assume $\omega\in \mathcal{ST}_J(X)$ for some complex structure $J$ so that $\varphi_0:=\text{Per}(J)\in \Phi_{[\omega]}$ by Lemma \ref{lem:kahler and tame period subdomain}. By Lemma \ref{lem:Grassmannian}, one can connect $\varphi_0$ to some $\varphi_1\in M_\kappa$ by a smooth path $\{\varphi_s\}\subseteq \Phi_\kappa$. Choose a smooth path of complex structures $J_s$ such that $J_0=J$ and $\text{Per}(J_s)=\varphi_s$. By Lemma \ref{lem:plus or minus}, $[\omega]$ must belong to $\mathcal{K}_{J_s}^t(X)$ for every $s$. Hence, we may apply Lemma \ref{lem:isotopic} to obtain an isotopy from $\omega$ to a K\"ahler form with respect to $J_1$. By Moser's stability, $\omega$ is also of K\"ahler type. Thus, we see that $\omega\in\mathcal{SK}(X)$.

  (2) This follows immediately from (1) since $\mathcal{ST}(X)\hookrightarrow \mathcal{S}(X)$ is obviously open.

  (3) Assume $\omega_i\in\mathcal{SK}_{J_i}(X)\cap\mathcal{SK}_\kappa(X)$ for $i=0,1$. The periods $\varphi_i:=\text{Per}(J_i)$ belong to $M_\kappa$ by Lemma \ref{lem:kahler and tame period subdomain}. Lemma \ref{lem:plus or minus} then further implies they must live in the same connected component of $M_\kappa$ so that there is a smooth path $\{\varphi_s\}\subseteq \Phi_\kappa$ connecting $\varphi_0$ to $\varphi_1$. Moreover, the path can be chosen such that there exists some $\varphi_s$ such that $\langle\varphi_s,\dd\rangle\neq 0$ for all $\dd\in L$ with $\dd^2=-2$. The argument in Lemma \ref{lem:cpxconn} then shows that there is a path of complex structures $\{I_s\}$ with $\text{Per}(I_s)=\varphi_s$, $I_0=J_0$ and $I_1=f^*J_1$ for some $f\in \Dfh(X)$. Similar to (2), we may apply Lemma \ref{lem:plus or minus} and Lemma \ref{lem:isotopic} to obtain an isotopy between $\omega_0$ and $f^*\omega_1$. This implies that $\omega_0$ and $\omega_1$ can be related by some homologically trivial diffeomorphism by Moser's stability. Thus, by the definition of $\mathcal{MK}_\kappa(X)$, we have \(\#\mathcal{MK}_\kappa(X)= 1\).

\end{proof}

As in Remark~\ref{rmk: alt proof of connectedness}, one can have an alternative proof that
\(
\#\pi_0(\mathcal{MK}(X))=1,
\)
since we know
\(
\mathcal{ST}(X)=\mathcal{SK}(X)
\) from item~(2) above and the connectedness of $\widetilde{\mathcal{I}}(X)$ by Lemma \ref{lem:cpxconn}.

\subsection{Non-minimal case}
We now consider the non-minimal smooth manifolds
\[
X_k:=X\# k\overline{\mathbb{CP}}^2,
\]
where \(X\) is either the K3 manifold or \(\mathbb T^4\). We always use the notation $\kappa-\sum_{i=1}^k\lambda_ie_i$ to denote a class in $H^2(X_k;\RR)\cong L_\RR\oplus \ZZ\langle e_1\rangle\oplus\cdots \oplus \ZZ\langle e_k\rangle$, where $\kappa\in L_\RR$, $\lambda_i\in\RR$ and $e_i$ is the $i$-th exceptional class.

\subsubsection{Complex structures on blowups}\label{section:cpx str on blowup}

Let \(I\) be a complex structure on \(X_k\). By Castelnuovo's contraction theorem, we can successively blow down exceptional curves and obtain a smooth minimal complex surface \((S,J)\). Thus \((X_k,I)\) is obtained from \((S,J)\) by a sequence of holomorphic blowups
\[
(X_k,I)
\cong
\operatorname{Bl}_{x_l}\cdots \operatorname{Bl}_{x_1}(S,J),
\]
where
\[
x_i\in S_{i-1},
\qquad
S_i=\operatorname{Bl}_{x_i}S_{i-1},
\qquad
S_0=S.
\]
By the smooth invariance of the holomorphic Kodaira dimension \(\kappa\) (\cite{FQ94}), and since \(X_k\) is diffeomorphic to a blowup of either a K3 surface or a complex two-torus, we have
\(
\kappa(X_k,I)=0.
\)
Since Kodaira dimension is also a birational invariant, it follows that
\(
\kappa(S,J)=0.
\)
Moreover, blowing down does not change \(b_1\) or \(b_2^+\). Therefore \(S\) has the same values of \(b_1\) and \(b_2^+\) as \(X\). By the Enriques--Kodaira classification of compact minimal complex surfaces with \(\kappa=0\), the minimal model \((S,J)\) is a K3 surface if \(X\) is the K3 manifold, and is a complex two-torus if \(X=\mathbb T^4\). In particular, we have $k=l$ and \(S\) is diffeomorphic to \(X\).
Thus every complex structure on \(X_k\) is obtained from a complex structure on \(X\) by successively blowing up an ordered \(k\) possibly infinitely near points. Consequently, the relevant parameters are a complex structure $J$ on the minimal surface \(X\), equivalently its period $\varphi=\text{Per}(J)\in\Phi$, together with a tuple of blowup centers on $X_i$ for each $i\leq k-1$. 

Let
\(
\pi:\mathcal X\to B
\)
be a smooth proper holomorphic family of complex surfaces whose fibers are
diffeomorphic to \(X\). We explain how a deformation of complex structures on the
minimal surface naturally induces a deformation of complex structures on its
blowups. We also allow the blowup centers to vary, including infinitely near points.

First consider the one-point blowup. Let
\[
\mathcal X\times_B\mathcal X
=
\{(x_1,x_2)\in \mathcal X\times \mathcal X
\mid
\pi(x_1)=\pi(x_2)\}
\]
be the fiber product, and let
\[
\Delta_{\mathcal X/B}
=
\{(x,x)\mid x\in \mathcal X\}
\subseteq
\mathcal X\times_B\mathcal X
\]
be its relative diagonal. The relative diagonal is a smooth complex
codimension-two submanifold, with normal bundle naturally isomorphic to the
relative tangent bundle \(T_{\mathcal X/B}\). We can blowup $\Delta_{\mathcal{X}/B}$ to obtain
\[
\mathcal X_1
:=
\operatorname{Bl}_{\Delta_{\mathcal X/B}}
(\mathcal X\times_B\mathcal X).
\]
Let
\(
p_2:\mathcal X\times_B\mathcal X\to \mathcal X
\)
be the projection to the second factor. Composing the blowup map with \(p_2\)
gives a smooth proper holomorphic map
\(
\pi_1:\mathcal X_1\to \mathcal X.
\)
For a point \(x\in \mathcal X_b:=\pi^{-1}(b)\), the fiber of \(\pi_1\) over
\(x\) is canonically isomorphic to
\(
\operatorname{Bl}_x(\mathcal X_b).
\)
Thus \(\pi_1:\mathcal X_1\to \mathcal X\) is a holomorphic
family with fibers diffeomorphic to $X_1$. Its base \(\mathcal X\) parametrizes pairs consisting of a fiber
\(\mathcal X_b\) and a point \(x\in \mathcal X_b\).

We now iterate this construction for multiple-point blowup. Set
\[
B_0:=B,\qquad
\mathcal X_0:=\mathcal X,\qquad
\pi_0:=\pi
.\] Suppose that we have constructed a smooth proper holomorphic family
\(
\pi_{i-1}:\mathcal X_{i-1}\to B_{i-1},
\)
whose fibers are diffeomorphic to $X_{i-1}$. Define
\(
B_i:=\mathcal X_{i-1}.
\)
Thus a point of \(B_i\) is a choice of a point on a fiber of
\(\mathcal X_{i-1}\to B_{i-1}\), namely a choice of the \(i\)-th blowup center.
Consider the fiber product
\(
\mathcal X_{i-1}\times_{B_{i-1}}\mathcal X_{i-1}
\)
and its relative diagonal
\(
\Delta_i
:=
\Delta_{\mathcal X_{i-1}/B_{i-1}}
\subseteq
\mathcal X_{i-1}\times_{B_{i-1}}\mathcal X_{i-1},
\)
which is still a smooth complex codimension-two submanifold. Define
\[
\mathcal X_i
:=
\operatorname{Bl}_{\Delta_i}
\bigl(
\mathcal X_{i-1}\times_{B_{i-1}}\mathcal X_{i-1}
\bigr).
\]
Projection to the second factor gives a smooth proper holomorphic map
\[
\pi_i:\mathcal X_i\to B_i=\mathcal X_{i-1}.
\]
The fiber of \(\pi_i\) over a point
\(
x_i\in (\mathcal X_{i-1})_{b_{i-1}}
\)
is the blowup of the fiber \((\mathcal X_{i-1})_{b_{i-1}}\) at \(x_i\).
Inductively, a point of \(B_k\) records a point $b\in B$ together with a sequence of blowup centers
\[
x_1\in X_0,\qquad
x_2\in X_1,\qquad
\ldots,\qquad
x_k\in X_{k-1},
\]
where
\[
X_0=\mathcal X_b,
\qquad
X_i=\operatorname{Bl}_{x_i}X_{i-1}.
\]
Therefore the fiber of
\(
\pi_k:\mathcal X_k\to B_k
\)
over such a point is naturally isomorphic to the iterated blowup
\(
\operatorname{Bl}_{x_k}\cdots \operatorname{Bl}_{x_1}(\mathcal X_b).
\)
The centers \(x_i\) are allowed to be infinitely near, because \(x_i\) is chosen
on the already blown-up surface \(X_{i-1}\), rather than necessarily on the
original surface \(X_0\).

In particular, the above construction, together with
Lemma~\ref{lem:cpxconn}, shows that
\(
{\mathcal I}(X_k)
\)
is connected. Applying Theorem \ref{thm:connected}, we obtain an
affirmative answer to Question~\ref{question:connected} for $X_k$.

\begin{corollary}\label{cor:blowup connectedness}
The space $\mathcal{MK}(X_k)$ is connected.
\end{corollary}

Next, we define an open subspace \(B_k^*\subseteq B_k\) corresponding to blowups
at \(k\) distinct points. 
Let $E_l\subseteq B_l=\mathcal{X}_{l-1}$ be the exceptional
divisor of the blowup
\(
\mathcal X_{l-1}\to \mathcal{X}_{l-2}\times_{B_{l-2}} \mathcal{X}_{l-2}
\) for every $l\geq 2$.
Set
\(
B_1^*:=B_1=\mathcal X_0.
\)
Suppose \(B_{i-1}^*\subseteq B_{i-1}\) has been defined. 
We further define \(B_i^*\subseteq B_i\)
to be
\[B_i^*:=(B_i\setminus E_i)\bigcap \pi^{-1}_{i-1}(B_{i-1}^*).\]
Equivalently, a point of \(B_i^*\) consists of the data
\(
b\in B, x_1\in X_0, x_2\in X_1, \cdots, x_i\in X_{i-1}
\)
such that, for each \(2\le j\le i\), the point \(x_j\) does not lie on any exceptional curves created by the preceding blowups. Thus all blowup centers descend to pairwise distinct points
of the original surface \(X_0\).
Hence, we may consider
\[
\mathcal X_k^*
:=
\pi_k^{-1}(B_k^*)\subseteq \mathcal X_k,\qquad \pi_k|_{\mathcal X_k^*}:\mathcal X_k^*\to B_k^*,
\]
which is a subfamily whose fibers are biholomorphic to the blowups of the original fibers at
\(k\) distinct points.

In particular, the above construction shows that a deformation of complex structures
on the minimal surface \(\mathcal X_b\), together with a deformation of an
ordered \(k\) possibly infinitely near blowup centers, induces a
smooth proper holomorphic family of complex structures on the \(k\)-fold blowup.
After choosing a smooth trivialization of the underlying differentiable family,
this gives a smooth family of complex structures on the fixed smooth manifold $X_k$.

There is also a natural projection
\(
\Pi_k:\mathcal X_k\to B
\)
defined by composing the maps in the tower:
\[
\Pi_k
:=
\pi_0\circ \pi_1\circ \cdots\circ \pi_k.
\]
For each \(b\in B\), the fiber \(\Pi_k^{-1}(b)\) is naturally identified with
the deformation family of \(k\)-fold blowups of the fixed surface
\(
\mathcal X_b=\pi_0^{-1}(b).
\)
In other words, \(\Pi_k^{-1}(b)\) is the space obtained by applying the same
recursive construction to the case where the
base \(B\) is a single point $\{b\}$. This leads to the following simple observation which will be used later.

\begin{lemma}\label{lem: moving blowup point}
    For any subset \(A\subseteq B\), the subset
    \(
        \Pi_k^{-1}(A)\cap B_{k+1}^*
    \)
    is dense in \(\Pi_k^{-1}(A)\). If \(A\) is connected, then
    \(
        \Pi_k^{-1}(A)\cap B_{k+1}^*
    \)
    is also connected.
\end{lemma}

\begin{proof}
For each \(a\in A\), by construction,
\(
B_{k+1}^*\cap \Pi_k^{-1}(a)
\)
is the locus of ordered \((k+1)\)-blowup centers which descend to pairwise
distinct points on the original surface \(\mathcal X_a\). This locus is obtained
from \(\Pi_k^{-1}(a)\) by removing finitely many submanifolds of real
codimension at least \(2\). Hence, by Lemma \ref{lemma: connected complement},
it is connected and dense in \(\Pi_k^{-1}(a)\).

Since
\(
\Pi_k^{-1}(A)
=
\bigcup_{a\in A}\Pi_k^{-1}(a),
\)
it follows fiberwise that
\(
\Pi_k^{-1}(A)\cap B_{k+1}^*
\)
is dense in \(\Pi_k^{-1}(A)\).
Now assume that \(A\) is connected. Since the restriction
\(
\Pi_k:\Pi_k^{-1}(A)\cap B_{k+1}^*\rightarrow A
\)
is locally trivial, and its fibers are connected by the previous paragraph, we have the total space
\(
\Pi_k^{-1}(A)\cap B_{k+1}^*
\)
is also connected.
\end{proof}

\subsubsection{Proof of other main results}
 Now, we consider more regions in the period domain $\Phi$. Given any $\kappa\in L_\RR$ and positive real number $\lambda$, let us define 
    \[\Phi_{\kappa, \lambda}:=\{\varphi \in \Phi_\kappa\,|\, Z_{\varphi}(\kappa)^2 > \lambda^2\}.\]

 \begin{lemma}\label{lem:Grassmannian for blowup}
If $\kappa^2>\lambda^2$, then $\Phi_{\kappa,\lambda}$ is an open subset of $\Phi_\kappa$ whose intersection with each connected component of $\Phi_\kappa$ is connected.
\end{lemma}  
\begin{proof}
Clearly, $\Phi_{\kappa,\lambda}$ is an open subset of $\Phi_\kappa$, since the inequality
\(
Z_\varphi(\kappa)^2>\lambda^2
\)
is strict.
To prove the connectedness statement, let
\[
\Psi_{\kappa,\lambda}^{>0}:=\{\varphi\in\Psi_\kappa^{>0}\mid Z_\varphi(\kappa)^2>\lambda^2\}.
\]
Consider the open $2$-disk bundle
\(
\pi:\Psi_\kappa^{>0}\rightarrow\overline{M}_\kappa
\)
constructed in Lemma \ref{lem:Grassmannian}, and let
\(
\pi_\lambda:\Psi_{\kappa,\lambda}^{>0}\rightarrow\overline{M}_\kappa
\)
be its restriction. Since $\kappa^2>\lambda^2$, we have
\(
Z_\varphi(\kappa)^2=\kappa^2>\lambda^2
\)
for every $\varphi\in\overline{M}_\kappa$. Hence,
\(
\overline{M}_\kappa\subset\Psi_{\kappa,\lambda}^{>0},
\)
so $\pi_\lambda$ is still surjective.
Moreover, for every $\varphi\in\overline{M}_\kappa$, the fiber
$\pi_\lambda^{-1}(\varphi)$ is an open subdisk of the fiber $\mathbb D_\varphi$
containing the center point $\varphi$. It follows that
$\Psi_{\kappa,\lambda}^{>0}$ is connected within each connected component of
$\Psi_\kappa^{>0}$.
When $X$ is a K3 surface,
\(
\Phi_{\kappa,\lambda}
=\Psi_{\kappa,\lambda}^{>0}
\setminus
\bigcup_{\delta\in\Delta_\kappa}\overline{M}_\delta,
\)
whereas for $X=\mathbb T^4$,
\(
\Phi_{\kappa,\lambda}=\Psi_{\kappa,\lambda}^{>0}.
\)
Therefore, Lemma \ref{lemma: connected complement} implies that the intersection of
$\Phi_{\kappa,\lambda}$ with each connected component of $\Phi_\kappa$ is also connected.
    \end{proof}  

Let us further define
 $$ \mathring{\Phi}_{\kappa, \lambda}:= {\Phi}_{\kappa, \lambda}\cap  \mathring{\Phi},
$$
where $\mathring{\Phi}$ is the subset of $\Phi$ by removing walls determined by all nonzero integral classes:
    \[\mathring{\Phi}:=\{\varphi\in \Phi\,|\,\langle\varphi,\dd\rangle\neq 0\text{ for any }\dd\in L\setminus\{0\}\}.\]

\begin{lemma} \label{lem: mathring Grassmannian}
$\mathring{\Phi}_{\kappa, \lambda}$ is a dense subset in  $\Phi_{\kappa, \lambda}$ whose intersection with each connected component of $\Phi_{\kappa,\lambda}$ is connected. 
\end{lemma}
\begin{proof}
The density statement follows from the Baire category theorem; the connectedness statement follows from Lemma \ref{lemma: connected complement} and \ref{lem:Grassmannian for blowup}.
\end{proof}

\begin{rmk}
Although \(\mathring{\Phi}_{\kappa,\lambda}\) is dense in
\(\Phi_{\kappa,\lambda}\), it is not necessarily open, since it is obtained by
removing the walls associated to all nonzero integral classes, and this wall
arrangement need not be locally finite. On the other hand,
\(M_\kappa\subseteq \overline{M}_\kappa\) is known to be open. Indeed, by
\cite[Lemma 5]{SmirnovK3}, the union
\(
\bigcup_{\delta\in\Delta_\kappa}
\left(\overline{M}_\kappa\cap \overline{M}_\delta\right)
\)
is locally finite in \(\overline{M}_\kappa\). Hence it is closed, and its
complement \(M_\kappa\) is open.
\end{rmk}

\begin{lemma}\label{lem: I-Kahler}
Let $(\omega,I)$ be a holomorphically tamed pair on $X_k$ with
\(
[\omega]=\kappa-\sum_{i=1}^k\lambda_i e_i.
\)
Suppose $(X_k,I)$ has a minimal model $(X,J)$, where
\(
\varphi:=\operatorname{Per}(J)\in\Phi.
\)
If
\(
\lambda^2\le \sum_{i=1}^k\lambda_i^2,
\)
then
\(
\varphi\in\Phi_{\kappa,\lambda}.
\)
\end{lemma}

    \begin{proof}
  The proof is similar to Lemma \ref{lem:kahler and tame period subdomain}. By Theorem \ref{thm:LZb1}, the projection $[\omega]_I^{1,1}$ of $[\omega]$ into the $H_I^{1,1}$-direction must be an $I$-K\"ahler class. Notice that the exceptional class $e_i$ must be a $(1,1)$-class for every $i$. Hence, $[\omega]_I^{1,1}=Z_\varphi(\kappa)-\sum_{i=1}^k\lambda_i e_i\in\mathcal{K}_I^c(X_k)$. In particular, \[0<[\omega]_I^{1,1}\cdot [\omega]_I^{1,1}=Z_\varphi(\kappa)^2-\sum_{i=1}^k\lambda_i^2\leq Z_\varphi(\kappa)^2-\lambda^2\] so that we have $\varphi\in \Phi_{\kappa,\lambda}$.
    \end{proof}

\begin{lemma} \label{lem: union of generic tamed cones}
Let $I$ be a complex structure on $X_k$. Suppose that $(X_k,I)$ is obtained from blowing up $k$ distinct points on its minimal model $(X,J)$ with exceptional curves representing classes $e_1,\cdots,e_k$. Assume that $\text{Per}(J)\in\mathring{\Phi}_{\kappa,\lambda}$ and $\kappa\in \mathcal{K}_J^t(X)$. If $\lambda_1,\cdots,\lambda_k$ are positive numbers such that
\(
\lambda^2\ge \sum_{i=1}^k\lambda_i^2,
\)
then 
     \(\kappa-\sum_{i=1}^k\lambda_i e_i\in \mathcal K_I^t(X_k)\).  
\end{lemma}

\begin{proof}
Since $\operatorname{Per}(J)\in\mathring{\Phi}$, the minimal surface $(X,J)$ contains no effective curves. Consequently, the only irreducible curves on $(X_k,I)$ are the exceptional curves by our assumption that the blowups occur at $k$ distinct points.
The exceptional classes $e_1,\dots,e_k$ are of type $(1,1)$ with respect to $I$. Therefore, the $(1,1)$-component of
\(
\kappa-\sum_{i=1}^k\lambda_i e_i
\)
is
\(
Z_\varphi(\kappa)-\sum_{i=1}^k\lambda_i e_i.
\)
Since $\text{Per}(J)\in\Phi_{\kappa,\lambda}$,
\(
Z_\varphi(\kappa)^2>\lambda^2,
\)
and hence
\[
(Z_\varphi(\kappa)-\sum_{i=1}^k\lambda_i e_i)^2
=
Z_\varphi(\kappa)^2-\sum_{i=1}^k\lambda_i^2
>
\lambda^2-\sum_{i=1}^k\lambda_i^2
\ge0.
\]
Moreover,
\[
(Z_\varphi(\kappa)-\sum_{i=1}^k\lambda_i e_i)\cdot e_i
=\lambda_i>0
\]
for every exceptional class $e_i$. Since these are the only irreducible curves on $(X_k,I)$, the Nakai--Moishezon criterion implies that
\(
Z_\varphi(\kappa)-\sum_{i=1}^k\lambda_i e_i
\in
\mathcal K_I^c(X_k).
\)
Finally,
\[
\kappa-\sum_{i=1}^k\lambda_i e_i
=
(
Z_\varphi(\kappa)-\sum_{i=1}^k\lambda_i e_i
)
+
\left(
\kappa-Z_\varphi(\kappa)
\right),
\]
where
\[
\kappa-Z_\varphi(\kappa)
\in\bigl(H^{2,0}_J(X)\oplus H^{0,2}_J(X)\bigr)_\mathbb R\subseteq
\bigl(H^{2,0}_I(X_k)\oplus H^{0,2}_I(X_k)\bigr)_\mathbb R.
\]
Therefore, Theorem \ref{thm:LZb1} yields
\(
\kappa-\sum_{i=1}^k\lambda_i e_i
\in
\mathcal K_I^t(X_k).
\)
\end{proof}

\begin{corollary}\label{cor:tame cone=symp cone}
     The integrable tamed cone $\mathcal K_{\operatorname{Int}}^t(X_k)$ is the complement of $\bigcup_{i=1}^k e_i^\perp$ in the positive cone. More explicitly,
    \[
    \mathcal K_{\operatorname{Int}}^t(X_k)
    =
    \{
    \kappa-\sum_{i=1}^k\lambda_i e_i
    \,|\,
    \kappa^2>\sum_{i=1}^k\lambda_i^2,\ \lambda_i\neq 0
   \}.
    \]
\end{corollary}

\begin{proof}
    Let \( a=\kappa-\sum_{i=1}^k\lambda_i e_i \). If \( \kappa^2>\sum_{i=1}^k\lambda_i^2\) and \(\lambda_i\neq 0 \) for every $i$. Set \( \lambda=(\sum_{i=1}^k\lambda_i^2)^{1/2}. \) Then $\kappa^2>\lambda^2$. Choose \( \varphi\in\mathring{\Phi}_{\kappa,\lambda} \) and a complex structure $J$ on $X$ with $\operatorname{Per}(J)=\varphi$ such that \( \kappa\in\mathcal K_J^t(X). \) Blowing up $k$ distinct points of $(X,J)$ and applying Lemma \ref{lem: union of generic tamed cones}, we obtain a complex structure $I$ on $X_k$ such that \( \kappa-\sum_{i=1}^k|\lambda_i|e_i \in\mathcal K_I^t(X_k). \) For each $i$, complex conjugation on the $i$-th $\overline{\mathbb{CP}}^{2}$-summand extends to an orientation-preserving self-diffeomorphism of $X_k$ whose induced action sends $e_i$ to $-e_i$ and fixes $e_i^\perp$. By composing such diffeomorphisms, we obtain an orientation-preserving self-diffeomorphism $f$ satisfying \( f^*e_i=\operatorname{sgn}(\lambda_i)e_i \) for every $i$, while preserving the $H^2(X;\RR)$-summand. Consequently, \( f^*(\kappa-\sum_{i=1}^k\lambda_i e_i) = \kappa-\sum_{i=1}^k|\lambda_i|e_i. \) Pulling back the holomorphically tamed pair by $f^{-1}$ shows that \( a\in\mathcal K_{\operatorname{Int}}^t(X_k). \) Conversely, if $a$ is represented by a symplectic form, then $a^2>0$. Moreover, the Seiberg--Witten blowup formula and SW=Gr theorem by Taubes imply that \(a\cdot e_i\neq 0 \) for every $i$. Thus, the integrable tamed cone is precisely the complement of $\bigcup_{i=1}^k e_i^\perp$ in the positive cone.
\end{proof}

We are now ready to state and prove our main results for the non-minimal manifold $X_k$ in different aspects.

\begin{theorem}\label{thm:blowup uniqueness}
Let $\omega_1,\omega_2$ be two cohomologous holomorphically tamed symplectic forms on $X_k$. Then there exists $f\in\Dfh(X_k)$ such that $f^*\omega_2=\omega_1$.
\end{theorem}

    \begin{proof}    
 Choose complex structures $I_1$ and $I_2$ on $X_k$ tamed by $\omega_1$ and $\omega_2$, respectively. We first show that they can be perturbed to complex structures $I_1'$ and $I_2'$, still tamed by $\omega_1$ and $\omega_2$, whose minimal models contain no curves.
Suppose that $(X_k,I_i)$ has minimal model $(X,J_i)$ with period
\(
\varphi_i=\operatorname{Per}(J_i)\in\Phi\) for \( i=1,2.
\)
Write
\(
[\omega_1]=[\omega_2]=\kappa-\sum_{j=1}^k\lambda_je_j\in H^2(X_k;\RR),
\)
and let
\(
\lambda=(\sum_{j=1}^k\lambda_j^2)^{\frac12}.
\)
By Lemma \ref{lem: I-Kahler}, we have $\varphi_i\in\Phi_{\kappa,\lambda}$.
By the local Torelli theorem, there exists a deformation
\(
\mathcal X^i\to B^i
\)
of $(X,J_i)$ such that $B^i$ is identified with a neighborhood of $\varphi_i$ in $\Phi$. Consider the corresponding $k$-fold blowup family constructed in Section \ref{section:cpx str on blowup},
\(
\pi_k^i:\mathcal X_k^i\rightarrow B_k^i=\mathcal X_{k-1}^i,
\)
and let $b_k^i\in B_k^i$ denote the point corresponding to $(X_k,I_i)$. This means the fiber
\(
(\pi_k^i)^{-1}(b_k^i)
\)
is biholomorphic to $(X_k,I_i)$.
Since $\mathring{\Phi}_{\kappa,\lambda}$ is dense in $\Phi_{\kappa,\lambda}$ by Lemma \ref{lem: mathring Grassmannian}, we may choose a point
\(
(b_k^i)'\in B_k^i
\)
arbitrarily close to $b_k^i$ such that
\(
\Pi_{k-1}\bigl((b_k^i)'\bigr)\in\mathring{\Phi}_{\kappa,\lambda}.
\)
Let $I_i'$ denote the complex structure on $X_k$ corresponding to $(b_k^i)'$. Since the taming condition is open, $I_i'$ is still tamed by $\omega_i$, provided $(b_k^i)'$ is chosen sufficiently close to $b_k^i$.
Replacing $I_i$ by $I_i'$, we may therefore assume from the outset that the minimal model of $(X_k,I_i)$ has period $\varphi_i$ in $\mathring{\Phi}_{\kappa,\lambda}$.
  
Now apply the connectedness statement in Lemma \ref{lem: mathring Grassmannian}, together with Lemma \ref{lem:plus or minus}, to choose a path
\(
\{\varphi_t\}\subseteq\mathring{\Phi}_{\kappa,\lambda}
\)
connecting $\varphi_1$ and $\varphi_2$. By Lemma \ref{lem:cpxconn}, after composing with a homologically trivial diffeomorphism if necessary, we may assume that $(X,J_1)$ and $(X,J_2)$ lie in the same deformation family
\(
\pi:\mathcal X\rightarrow B,
\)
where the base $B$ contains a path $\{b(t)\}$ whose period map is precisely $\{\varphi_t\}$.

Let $\beta_1,\beta_2\in B_k$ denote the points corresponding to $(X_k,I_1)$ and $(X_k,I_2)$, respectively, in the associated $k$-fold blowup family. By the density statement of Lemma \ref{lem: moving blowup point}, we may deform each $\beta_i$ to a nearby point
\(
\beta_i'\in \Pi_{k-1}^{-1}(\{b(t)\})\cap B_k^*,
\)
such that the corresponding complex structure on $X_k$ remains tamed by $\omega_i$.
The connectedness statement of Lemma \ref{lem: moving blowup point} then provides a path
\(
\{\gamma(t)\}\subseteq
\Pi_{k-1}^{-1}(\{b(t)\})\cap B_k^*
\)
connecting $\beta_1'$ to $\beta_2'$. Let $(X_k,I(t))$ denote the corresponding family of complex structures. Then Lemma \ref{lem: union of generic tamed cones} implies that
\(
[\omega_1]=[\omega_2]\in\mathcal K_{I(t)}^t(X_k)
\)
for every $t$. Finally, we again use Lemma \ref{lem:isotopic} together with Moser's stability to obtain that $\omega_1$ and $\omega_2$ are related by some $f\in\Dfh(X_k)$.
    \end{proof}

As an immediate consequence of the preceding theorem, Questions~\ref{question:sympST} and~\ref{question:unique} have affirmative answers for $X_k$.

\begin{corollary}\label{cor:uniqueness on blowup}
For any $a\in H^2(X_k;\RR)$, the following hold:
\begin{enumerate}
\item $\mathcal{SK}_a(X_k)$ is either empty or equal to $\mathcal{ST}_a(X_k)$. Equivalently, every holomorphically tamed symplectic form on $X_k$ representing a class in $\mathcal{K}_{\text{Int}}^c(X_k)$ is of K\"ahler type.

\item $\#\mathcal{MK}_a(X_k)\leq 1$. Equivalently, any two cohomologous symplectic forms of K\"ahler type on $X_k$ are related by a homologically trivial diffeomorphism.

\end{enumerate}
\end{corollary}

\begin{proof}
Every symplectic form of K\"ahler type is holomorphically tamed. In particular, Theorem \ref{thm:blowup uniqueness} applies to any symplectic form of K\"ahler type, giving the desired conclusion.
\end{proof}

\begin{rmk}
Unlike in the minimal case, the preceding corollary does not imply that
\(
\mathcal{ST}(X_k)=\mathcal{SK}(X_k),
\)
since the class $a$ is assumed to lie in $\mathcal{K}_{\operatorname{Int}}^c(X_k)$. Indeed, Corollary \ref{cor:tame cone=symp cone} shows that the integrable tamed cone of $X_k$ coincides with its symplectic cone. On the other hand, \cite[Corollary~1.3]{LMS13} shows that, when $X=\mathbb T^4$, the one-point blowup $X_1$ admits symplectic forms that are not of K\"ahler type. Thus,
\(
\mathcal{ST}(X_1)\neq\mathcal{SK}(X_1)
\) for the blowup of $\mathbb T^4$.
\end{rmk}

%e will use the following version of Kodaira--Spencer stability theorem, which can be found in \cite[Theorem 5.6]{EV16}.
%\begin{theorem}[Kodaira--Spencer stability  \cite{KSstab,EV16}]\label{thm:KSstab}
%    Let $(M,J,\omega)$ be a closed K\"ahler manifold, and let $\{J_t\}$, $t\in B$, $J_0 = J$, be a smooth
%local deformation of $J$. Then there exists a neighborhood of $U\subseteq B$ of zero in $B$ such that $[\omega]_{J_t}^{1,1}$, the $(1,1)$-part of $[\omega]$ with respect to $J_t$, is in $\mathcal{K}_{J_t}^c(M)$ for all $t \in U$. 
%\end{theorem}
We next answer Question \ref{question:open} for $X_k$.

\begin{theorem}\label{thm:blowup openness}
 The space $\mathcal{SK}(X_k)$ of symplectic forms of K\"ahler type is open in the space $\mathcal S(X_k)$ of all symplectic forms.
\end{theorem}

\begin{proof}
We first show that the integrable compatible cone
\(
\mathcal K_{\operatorname{Int}}^c(X_k)
\)
is open in \(H^2(X_k;\RR)\).
Let
\(
a=\kappa-\sum_{i=1}^k\lambda_i e_i
\in\mathcal K_I^c(X_k),
\)
and choose an \(I\)-K\"ahler form \(\omega\) representing \(a\).
We aim to construct an open neighborhood of \(a\in H^2(X_k;\RR)\) that is also contained in
\(\mathcal K_{\operatorname{Int}}^c(X_k)\).

 Suppose that the minimal model of \((X_k,I)\) is \((X,J)\), with
\(
\varphi=\operatorname{Per}(J)\in\Phi.
\) Since \(\kappa\in H_J^{1,1}(X;\RR)\), the oriented positive-definite $2$-plane
\(P(\varphi)\) is contained in \(\kappa^\perp\). For every \(\kappa'\)
sufficiently close to \(\kappa\), the orthogonal projection of
\(P(\varphi)\) onto \((\kappa')^\perp\) is again an oriented positive-definite
\(2\)-plane. Denote the corresponding period point by
\(\varphi_{\kappa'}\). After shrinking to a neighborhood of \(\kappa\),
the assignment
\(
\kappa'\mapsto\varphi_{\kappa'}
\)
is smooth and satisfies
\(
\kappa'\in H_{J_{\kappa'}}^{1,1}(X;\RR)
\) for any $J_{\kappa'}\in\text{Per}^{-1}(\varphi_{\kappa'})$.

By the local Torelli theorem, we may choose a deformation
\(
\pi:\mathcal X\rightarrow B
\)
of \((X,J)\) such that the period map identifies \(B\) with a
neighborhood of \(\varphi\) in \(\Phi\). Consider the associated
\(k\)-fold blowup deformation family constructed in
Section~\ref{section:cpx str on blowup},
\(
\pi_k:\mathcal X_k\rightarrow B_k.
\)
Let \(\beta\in B_k\) be the point corresponding to \((X_k,I)\). Since
the natural projection
\(
\Pi_{k-1}:B_k\rightarrow B
\)
is a submersion, after shrinking \(B\) we may choose a smooth local
section
\(
s:B\rightarrow B_k
\)
through \(\beta\).

Now choose a sufficiently small neighborhood \(U\) of \(a\) in
\(H^2(X_k;\RR)\). For
\(
t=\kappa'-\sum_{i=1}^k\lambda_i'e_i\in U,
\)
let \(I_t\) be the complex structure on \(X_k\) corresponding to
\(
s(\varphi_{\kappa'})\in B_k.
\)
Under the natural smooth trivialization of the family, the
exceptional classes \(e_1,\ldots,e_k\) remain of type \((1,1)\).
Together with the choice of \(\varphi_{\kappa'}\), this gives
\(
t\in H_{I_t}^{1,1}(X_k;\RR)
\)
for every \(t\in U\).

We next adapt an argument from \cite[Theorem~5.6]{EV16}. By the
Kodaira--Spencer stability theorem \cite[Theorem~15]{KSstab}, after
shrinking \(U\) if necessary, the complex structures \(I_t\) admit a
smooth family of K\"ahler metrics \(g_t\). By the parametric version of
Hodge theory, the \(g_t\)-harmonic representative \(\Omega_t\) of the
class \(t\) depends smoothly on \(t\). Since
\(
t\in H_{I_t}^{1,1}(X_k;\RR)
\)
and the Hodge Laplacian of a K\"ahler metric preserves type,
\(\Omega_t\) is a real \((1,1)\)-form with respect to \(I_t\).

Let \(\Omega_a\) be the harmonic representative of \(a\). Since
\(\Omega_a\) and \(\omega\) are cohomologous real \((1,1)\)-forms, the
\(\partial\overline{\partial}\)-lemma gives a smooth
function \(F\) such that
\(
\omega
=
\Omega_a
+
\sqrt{-1}\,\partial_I\overline{\partial}_I F.
\)
We can then extend \(F\) to a
smooth family of functions \(F_t\) for $t\in U$ and define
\(
\omega_t
:=
\Omega_t
+
\sqrt{-1}\,\partial_{I_t}\overline{\partial}_{I_t}F_t.
\)
Then
\(
\omega_t\) is still an $(1,1)$-form with respect to $I_t$ in the class $t$. Note that \(\omega_a=\omega\) is tamed with $I_a=I$. Since tameness is an open condition, after
shrinking \(U\) once more, \(\omega_t\) is \(I_t\)-K\"ahler for every
\(t\in U\). It follows that
\(
U\subseteq\mathcal K_{\operatorname{Int}}^c(X_k),
\)
and hence \(\mathcal K_{\operatorname{Int}}^c(X_k)\) is open in
\(H^2(X_k;\RR)\).

We now pass from cohomological openness to openness at the level of symplectic forms. Let
\(
\omega\in\mathcal{SK}(X_k).
\)
Choose a complex structure \(I\) on \(X_k\) with respect to which
\(\omega\) is K\"ahler. By the openness proved above, there exists an
open neighborhood
\(
U\subseteq H^2(X_k;\RR)
\)
of \([\omega]\) such that
\(
U\subseteq\mathcal K_{\operatorname{Int}}^c(X_k).
\)
By the continuity of the cohomology class map from
\(\mathcal S(X_k)\) to \(H^2(X_k;\RR)\) shown in
Lemma~\ref{lem: continuity of the class map}, there exists an open
neighborhood
\(
\mathcal U\subseteq\mathcal S(X_k)
\)
of \(\omega\) such that
\(
[\omega']\in U
\)
for every \(\omega'\in\mathcal U\).
After shrinking \(\mathcal U\), we may also assume that every
\(\omega'\in\mathcal U\) tames the fixed complex structure \(I\), since
tameness is an open condition. For each \(\omega'\in\mathcal U\), the
condition
\(
[\omega']\in\mathcal K_{\operatorname{Int}}^c(X_k)
\)
implies that the holomorphically tamed \(\omega'\) is of
K\"ahler type by Corollary \ref{cor:uniqueness on blowup}. Thus,
\(
\mathcal U\subseteq\mathcal{SK}(X_k),
\)
which proves that \(\mathcal{SK}(X_k)\) is open in
\(\mathcal S(X_k)\).
\end{proof}

\section{Proofs of theorems in Section \ref{section:main results}}\label{section:proof}
\begin{proof}[Proof of Theorem \ref{thm:main1}]
    Case (1) was proved in Theorem \ref{thm:b+=1 main} (1). Case (2) was proved in Theorem \ref{thm:minimal main} (1) and Corollary \ref{cor:uniqueness on blowup} (1). For case (3), note that if $\mathcal{SK}_a(X)\neq \varnothing$ and $\omega\in \mathcal{ST}_J(X)\cap\mathcal{ST}_a(X)$, then $a\in H_J^{1,1}(X;\RR)$ by Corollary \ref{cor:rigidity}. It then follows from Theorem \ref{thm:LZb1} that $a\in\mathcal{K}^c_{J}(X)$. Hence, there exists some $\omega'\in\mathcal{SK}_a(X)$. Moser's stability thus implies $\omega$ is diffeomorphic to $\omega'$ so that we also have $\omega\in \mathcal{SK}_a(X)$.
\end{proof}

\begin{proof}[Proof of Theorem \ref{thm:main2}]
    Case (1) was proved in Corollary \ref{cor:Enriques hyperellpitic uniqueness}, Theorem \ref{thm:minimal main} (3), Corollary \ref{cor:uniqueness on blowup} (2) and \cite[Theorem 3.18]{Li08}. Case (2) follows from Corollary \ref{cor:rigidity}. The finiteness result is Corollary \ref{cor:general type uniqueness}. 
\end{proof}

\begin{proof}[Proof of Theorem \ref{thm:main3}]
    The finiteness in (1) is part of Corollary \ref{cor:general type deformation finite}. Case (2) follows from Corollary \ref{cor:blowup connectedness}, the proof of Corollary \ref{cor:Enriques hyperellpitic uniqueness} and \cite[Theorem 3.5]{Li08}. Case (3) is part of Corollary \ref{cor:general type deformation finite}. 
\end{proof}

\begin{proof}[Proof of Theorem \ref{thm:main4}]
    Case (1) was proved in Theorem \ref{thm:b+=1 main} (2). Case (2) was proved in Theorem \ref{thm:minimal main} (2) and \ref{thm:blowup openness}. For the non-openness statement, notice that by Corollary \ref{cor:rigidity}, $\mathcal{K}_{\text{Int}}^c(X)$ must be contained in the subspace $H^{1,1}_J(X;\RR)\subseteq H^2(X;\RR)$ for any complex structure $J$ on $X$. If $b_2^+(X)>1$, this subspace is proper. It follows that the integrable compatible cone cannot be open in $H^2(X;\RR)$, while the symplectic cone does. Therefore, $\mathcal{SK}(X)$ cannot be open in $\cS(X)$.
\end{proof}

%\section{Higher dimensions}
%$T^{2n}$ and \cite{Wilson},\cite{Wilson97}, McDuff and Ruan

\bibliographystyle{amsalpha}
	\bibliography{main}{}

\end{document}